\documentclass{book}

\usepackage{makeidx}
\usepackage{amsthm,amsmath,amssymb,amsfonts}
\usepackage{setspace,graphicx}
\usepackage{Generic}  % keep the Generic.sty file in the same folder
\usepackage[sort,longnamesfirst]{natbib}
\usepackage{url} 
\usepackage{subcaption}
\usepackage{hyperref}
\usepackage{enumerate}

\numberwithin{equation}{section}
\theoremstyle{plain}
\newtheorem{theorem}{Theorem}[section]
\newcommand{\df}{\mathrm{d}}
\newcommand{\X}{\mathsf{X}}

\newcommand{\B}{\mathcal{B}}

\newcommand{\F}{\mathcal{F}}

\newcommand{\ind}{\mathbf{1}}
\newcommand{\tv}{\mbox{\scriptsize TV}}

\newcommand{\dist}{\psi}
\newcommand{\Dist}{D}

\newtheorem{remark}{Remark}[section]
\newtheorem{corollary}{Corollary}[section]
\newtheorem{example}{Example}[section]

\begin{document}
\doublespacing
%  The chapter command will create the title heading for your chapter
%  
%  Put the title in the braces { } and put a short ``running head'' 
%  version of the title (an abbreviated version of the title) to 
%  appear at the top of the left hand pages in the brackets

%  NOTE:  only the first word in the title starts with a capital letter!
%  Same for the short title

\chapter[Convergence Bounds]{Convergence Bounds for Monte Carlo Markov Chains} \label{sec: label}
\begin{center}
\begin{large}
{\em Qian Qin }
\end{large}
\end{center}

\section{Introduction} \label{sec:intro}

When implementing a Markov chain Monte Carlo (MCMC) algorithm, one simulates a Markov chain $(X_t)_{t=0}^{\infty}$ such that the distribution of $X_t$ approaches a target distribution $\varpi(\cdot)$ as $t$ increases.
Inference about the target distribution is then conducted based on a Monte Carlo sample, which is a finite portion of the chain $(X_t)_{t=\underline{n}}^{\overline{n}}$.
The integer $\underline{n}$ is the amount of burn-in, while the integer $\overline{n}$ is the time at which the simulation is terminated.
For the Monte Carlo sample $(X_t)_{t=\underline{n}}^{\overline{n}}$ to be representative of the target distribution, the distribution of most of the $X_t$'s, $\underline{n} \leq t \leq \overline{n}$, should be similar to $\varpi(\cdot)$.
To be certain of this, one would need to know how fast the distribution of $X_t$ converges to $\varpi(\cdot)$ as $t \to \infty$.
This chapter reviews several popular methods for convergence analysis, i.e., ascertaining the convergence properties of a Monte Carlo Markov chain using mathematics.
To be specific, we investigate ways to construct bounds on the distance between the distribution of a Markov chain at a given time point and the chain's stationary distribution.

Convergence analysis plays a pivotal role in the theory and application of MCMC.
Some important asymptotic results for MCMC estimators, such as the central limit theorem and the strong invariance principle, rely on conditions on the convergence rate of the Monte Carlo Markov chain; see, e.g., \cite{jones2004markov} and \cite{kuelbs1980almost}.
This type of condition can be verified through appropriate convergence bounds.
It is also possible to derive non-asymptotic error bounds for Monte Carlo estimators based on convergence bounds \citep[][Theorems 3.34 and 3.41]{rudolf2012explicit}. 

%In modern era, MCMC algorithms are often applied to distributions that are defined on high dimensional spaces, or associated with large data sets with many observations and/or features.
%Thus, there is an increasing need to construct sharp convergence bounds that scale well with the dimensionality parameters associated with the target distribution.

A class of methods related to convergence analysis is convergence diagnostics, which aims to assess the performance of an MCMC algorithm by scrutinizing its output \citep{gelman1992inference,brooks1998convergence,gelman2011inference,roy2020convergence}.
Convergence diagnostics can detect problems with an MCMC simulation, but they cannot prove that the simulation is generating a representative sample.
While convergence analysis is often mathematically challenging, it offers robust theoretical guarantees that diagnostics cannot provide.

There is an enormous body of literature devoted to the topic at hand.
This chapter serves primarily as an introductory guide for those embarking on further research in this area.
It is assumed that readers possess a moderate level of familiarity with the languages of measure theoretic probability and linear algebra within Hilbert spaces.

The rest of this chapter is organized as follows.
In Section \ref{sec:basic}, we lay out the basic concepts and notations.
In Section \ref{sec:coupling}, we review the coupling method for constructing convergence bounds.
In Section \ref{sec:l2}, we describe the $L^2$ framework for convergence analysis, with a focus on methods involving isoperimetric inequalities.
Finally, some other methods for constructing convergence bounds are listed in Section \ref{sec:others}.

\section{Basic Setup} \label{sec:basic}

We begin by setting up some notations.
Let $\X$ be a Polish (separable and complete) metric space with metric $\dist: \X^2 \to [0,\infty)$, and let $\B$ be its Borel $\sigma$ algebra.
Denote by $\mathcal{P}(\X)$ the collection of probability measures, or distributions, on $(\X,\B)$.
Let $(X_t)_{t=0}^{\infty}$ be a time-homogeneous Markov chain whose state space is~$\X$.
Let $K: \X \times \B \to [0,1]$ be its Markov transition kernel (Mtk) so that $P(X_{t+1} \in A \mid X_t = x) = K(x, A)$ for $x \in \X$ and $A \in \B$.
For $t \in \mathbb{N}_+ := \{1,2,\dots\}$, we can define the $t$-step Mtk of the chain, which is a function $K^t: \X \times \B \to [0,1]$ satisfying $P(X_t \in A \mid X_0 = x) = K^t(x, A)$.
Indeed, one simply let $K^1(x,A) = K(x,A)$, and $K^{t+1}(x,A) = \int_{\X} K^t(x,\df y) K(y, A)$.
For $\mu \in \mathcal{P}(\X)$, let $\mu K^t(\cdot)  = \int_{\X} \mu(\df x) K^t(x, \cdot)$ if $t \in \mathbb{N}_+$, and, by convention, let $\mu K^0 = \mu$.
Then $\mu K^t(\cdot)$ is the distribution of $X_t$ if $X_0 \sim \mu$.

Assume that $(X_t)_{t=0}^{\infty}$ has a stationary distribution $\varpi \in \mathcal{P}(\X)$, so that $\varpi K^t(\cdot) = \varpi(\cdot)$ for $t \in \mathbb{N} = \{0\} \cup \mathbb{N}_+$.
If this chain is associated with an MCMC algorithm targeting $\varpi(\cdot)$, then the distribution of $X_t$ should converge to $\varpi(\cdot)$ in some sense as $t \to \infty$.
When conducting a convergence analysis, we seek to understand how fast $\mu K^t(\cdot)$ approaches $\varpi(\cdot)$ as~$t$ grows for some class of initial distributions $\mu(\cdot)$.

To conduct a quantitative analysis, we need to define a distance function that quantifies the difference between two probability measures.
A common way to construct such a distance is as follows \citep{zolotarev1984probability, muller1997integral}.
Let~$\F$ be a collection of real measurable functions on~$\X$.
Let $\F'$ be a subset of $\mathcal{P}(\X)$ such that $\int_{\X} |f(x)| \, \mu(\df x) < \infty$ for each $f \in \F$ whenever $\mu \in \F'$. 
For $\mu, \nu \in \F'$, we define the ``integral probability metric"
\[
\|\mu - \nu\|_{\F} = \sup_{f \in \F} |\mu f - \nu f|,
\]
where $\mu f = \int_{\X} f \, \df \mu$.
One can check that, for $\mu, \nu, \omega \in \F'$,
\[
\|\mu - \nu\|_{\F} \leq \|\mu - \omega\|_{\F} + \|\omega - \nu\|_{\F}.
\]
Assume that~$\F$ is rich enough so that $\|\mu - \nu\|_{\F} = 0$ implies that $\mu(A) = \nu(A)$ for $A \in \B$.
Then $\|\mu - \nu\|_{\F}$ serves as a distance between~$\mu$ and~$\nu$.

Below we list some commonly used distances constructed in this manner.
\begin{enumerate}[(I)] 
	\item \label{en:tv} $\F$ is the set of functions~$f$ such that $\sup_{x \in \X} |f(x)| = 1/2$, and $\F' = \mathcal{P}(\X)$.
	Then $\|\mu - \nu\|_{\F}$ is the total variation distance between~$\mu$ and~$\nu$ for $\mu, \nu \in \F'$.
	The same goes if $\F$ is the set of measurable indicator functions.
	In this case, we write $\|\mu - \nu\|_{\F}$ as $\|\mu - \nu\|_{\tv}$.
	If $\mu$ and $\nu$ are absolutely continuous with respect to some $\sigma$-finite measure $\lambda$, then
	\begin{equation} \label{eq:tv}
	\|\mu - \nu\|_{\tv} = \frac{1}{2} \int_{\X} \left| \frac{\df \mu}{\df \lambda}(x) - \frac{\df \nu}{\df \lambda}(x) \right|  \lambda(\df x) = 1 - \int_{\X} \min \left\{ \frac{\df \mu}{\df \lambda}(x), \frac{\df \nu}{\df \lambda}(x) \right\}  \lambda(\df x).
	\end{equation}
	
	\item \label{en:wasser}
	$\F$ is the set of functions~$f$ such that $\sup_{x \neq y} |f(x) - f(y)| / \dist(x,y) = 1$, i.e., the set of functions whose Lipschitz constant is 1.
	$\F'$ is the set of probability measures~$\mu$ such that $\int_{\X} \dist(x_0, x) \, \mu(\df x) < \infty$ for some $x_0 \in \X$.
	Then by the Kantorovich-Rubinstein duality \citep[see, e.g.,][Theorem 5.10]{villani2008optimal}, $\|\mu - \nu\|_{\F}$ is the 1-Wasserstein distance between~$\mu$ and~$\nu$ induced by~$\dist$.
	In this case, we write $\|\mu - \nu\|_{\F}$ as $W_\dist(\mu,\nu)$.
	
%	The total variation distance can be loosely viewed as a 1-Wasserstein distance with $\dist(x,y) = \ind_{x\neq y}$, where $\ind_{x \neq y}$ is one if $x \neq y$ and zero otherwise.
	
	\item \label{en:l2} $\F$ is the set of functions~$f$ such that $\int_{\X} f(x)^2 \, \varpi(\df x) = 1$.
	$\F'$ is the set of probability measures~$\mu$ such that~$\mu$ is absolutely continuous with respect to~$\varpi$ (i.e., $\mu \ll \varpi$), and that
	\[
	\int_{\X} \left[\frac{\df \mu}{\df \varpi} (x) \right]^2  \varpi(\df x) < \infty.
	\]
	(It can be checked that, if $\mu \ll \varpi$, then the above display is equivalent to $\int_{\X} |f(x)| \, \mu(\df x)$ being bounded as $f$ varies in $\F$.)
	Then $\|\mu - \nu\|_{\F}$ is the $L^2$ distance between~$\mu$ and~$\nu$.
	In this case, we write $\|\mu - \nu\|_{\F}$ as $\|\mu - \nu\|_2$.
	One can show via the Cauchy-Schwarz inequality that the $L^2$ distance has a dual representation
	\begin{equation} \label{eq:l2-distance}
	\|\mu - \nu\|_2 = \sqrt{\int_{\X} \left[ \frac{\df \mu}{\df \varpi}(x) - \frac{\df \nu}{\df \varpi}(x) \right]^2 \varpi(\df x) }.
	\end{equation}
	
\end{enumerate}

Throughout this chapter, assume that~$\F'$ contains the stationary distribution $\varpi$.
We also assume that $\mu K \in \F'$ whenever $\mu \in \F'$, so that $\mu K^t \in \F'$ for $t \in \mathbb{N}$ whenever $\mu \in \F'$.
It can be shown that the two assumptions always hold in scenarios \eqref{en:tv} and \eqref{en:l2}; see, e.g., Lemma 22.1.3 of \cite{douc2004practical}.
In scenario \eqref{en:wasser}, the second assumption holds if, say, there exist a point $x_0 \in \X$ and finite constants $c_1$ and $c_2$ such that $\int_{\X} \dist(x_0, x') \, K(x, \df x') \leq c_1 \dist(x_0, x) + c_2$ for $x \in \X$.
All examples herein satisfy these two assumptions.

A central goal of convergence analysis is to construct bounds on $\|\mu K^t - \varpi\|_{\F}$ for at least some initial distribution $\mu \in \F'$ that is practically feasible.
We are mainly concerned with constructing upper bounds on $\|\mu K^t - \varpi\|_{\F}$, although lower bounds will also be touched on.
In particular, we shall focus on several methods that enable us to form convergence bounds of the form
\[
\|\mu K^t - \varpi\|_{\F} \leq C_{\mu} \rho^t, \quad t \in \mathbb{N}_+,
\]
where $C_{\mu} \in (0,\infty)$ is a function of the initial distribution~$\mu$, and $\rho$ is a constant in $[0,1)$.
This type of bound, among others, can be used to bound the $(\epsilon,\mu)$-mixing time, which is the smallest $t \in \mathbb{N}_+$ such that $\|\mu K^t - \varpi\|_{\F} \leq \epsilon$, where $\epsilon \in (0,\infty)$ is some prescribed level of tolerance.
Indeed, denoting the $(\epsilon,\mu)$-mixing time by $t_{\F}(\epsilon,\mu)$, the above bound would yield
\[
t_{\F}(\epsilon,\mu) \leq \left\lceil \frac{\log C_{\mu} - \log \epsilon}{- \log \rho} \right\rceil
\]
if $\rho \in (0,1)$, where $\lceil \cdot \rceil$ is the ceiling function.

The choice of the distance function would of course affect the outcome of one's analysis, but bounds in terms of one distance can often be translated to those in terms of another.
For instance, using \eqref{eq:tv}, \eqref{eq:l2-distance}, and Jensen's inequality, one can show that if $\df \mu/\df \varpi$ exists and is squared integrable with respect to $\varpi$, then
\[
2 \|\mu K^t - \varpi\|_{\tv} \leq \|\mu K^t - \varpi\|_2,
\]
so an upper bound on the right-hand-side upper bounds the left-hand-side as well.
See \cite{roberts1997geometric}, \cite{roberts2001geometric}, and \cite{kontoyiannis2012geometric} for additional details on the relationship between convergence bounds in the $L^2$ and total variation distances.
In Section \ref{ssec:oneshot}, we discuss how to translate a bound in terms of the 1-Wasserstein distance to one in terms of the total variation distance.

For the construction of upper bounds on $\|\mu K^t - \varpi\|_{\F}$, we review (a) the coupling method when $\|\cdot - \cdot\|_{\F}$ is the total variation distance or the 1-Wasserstein distance, and (b) the $L^2$ theory, especially techniques based on the conductance and isoperimetric inequalities, when $\|\cdot - \cdot\|_{\F}$ is the $L^2$ distance.
We will then describe a simple method for lower bounding the ``convergence rate" in the $L^2$ framework, which quantifies how slow a chain converges.
Some other important methods are listed with references at the end of the chapter.

To end Section \ref{sec:basic}, we give a couple running toy examples on which we will demonstrate several techniques for convergence analysis.

\begin{example} \label{ex:IMH}
	Let $\X = [0,1]$, and let~$\B$ be the Borel subsets of $[0,1]$.
	Let $s: [0,1] \to (0,\infty)$ be a positive continuous probability density function, 	and denote the corresponding distribution by $\pi_s(\cdot)$.
	Note that, due to continuity, $M_s := \sup_{x \in [0,1]} s(x) < \infty$.
	For $x, x' \in \X$, let $a_s(x,x') = \min\{1, s(x')/s(x)\}$.
	Let $(X_t)_{t=0}^{\infty}$ be a Markov chain such that, given $X_t$, the next state $X_{t+1}$ is generated using the following procedure: Draw $X'$ from the uniform distribution on $[0,1]$; with probability $a_s(X_t, X') $, set $X_{t+1} = X'$; with probability $1 - a_s(X_t, X')$, set $X_{t+1} = X_t$.
	Then $(X_t)_{t=0}^{\infty}$ is associated with an independent Metropolis Hastings algorithm targeting $\pi_s(\cdot)$.
	Its transition kernel is
	\[
	K_s(x, A) = \int_A a_s(x,x') \, \df x' + \left[ 1 - \int_0^1 a_s(x,x') \, \df x' \right] \ind_{x \in A}, \quad x \in [0,1], \, A \in \B,
	\]
	where $\ind_{x \in A}$ is 1 if $x \in A$ and 0 otherwise.
	It is well-known that chains associated with Metropolis Hastings algorithms are reversible with respect to their target distributions.
\end{example}

\begin{example} \label{ex:gaussian}
	Let $\X = \mathbb{R}^p$, where $p$ is a positive integer, and let~$\B$ be the Borel sets.
	Let $\alpha \in [0,1)$ be a constant.
	Define the Gaussian chain as a Markov chain $(X_t)_{t=0}^{\infty}$ such that, given $X_t = x \in \mathbb{R}^p$, $X_{t+1}$ follows the $\mbox{N}_p(\alpha x, \, (1-\alpha^2) I_p )$ distribution, where $\mbox{N}_p(m, V)$ means the $p$-variate normal distribution with mean~$m$ and variance~$V$, and $I_p$ is the $p \times p$ identity matrix.
	Its transition kernel is
	\[
	K_{p, \alpha}(x, A) = \int_A \frac{1}{[2\pi (1- \alpha^2)]^{p/2}} \exp \left[ - \frac{1}{2(1-\alpha^2)} \| x' - \alpha x\|^2 \right]  \df x', \quad x \in \mathbb{R}^p, \, A \in \B,
	\]
	where $\|\cdot\|$ is the Euclidean norm.
	This chain is reversible with respect to the $\mbox{N}_p(0, I_p)$ measure, which will be denoted by $\varpi_p(\cdot)$.
	Indeed, 
	\[
	\int_A \varpi_p(\df x) K_{p, \alpha}(x, B) = \frac{1}{(2\pi)^p (1-\alpha^2)^{p/2}} \int_{A \times B} \exp \left[ - \frac{ \|x'\|^2 + \|x\|^2 - 2 \alpha x^{\top} x' }{2(1-\alpha^2)} \right] \df x \, \df x'
	\]
	is a symmetric function of $A \in \B$ and $B \in \B$.
\end{example}

Due of the simplicity of the chains in these examples, their convergence properties are well-understood, but for illustrative purposes we will feign ignorance in most of our analyses.

%
%\subsection{Exercises}
%
%\begin{enumerate}
%	\item 
%	Assume that $x \mapsto \|\delta_x K - \varpi\|_{\F'}$ is measurable.
%	Prove that, for $\mu \in \F'$ and $t \in \mathbb{N}$,
%	\[
%	\|\mu K^t - \varpi\|_{\F} \leq \int_{\X} \|\delta_x K^t - \varpi\|_{\F} \, \mu(\df t).
%	\]
%	\item 
%	Consider the cases where $\|\cdot\|_{\F}$ is the total variation distance or the $L^2$ distance.
%	Prove that, for $\mu \in \F'$ and $t \in \mathbb{N}$, it holds that $\mu K^t \in \F'$, and $\|\mu K^{t+1} - \varpi\|_{\F} \leq \|\mu K^t - \varpi\|_{\F}$.
%\end{enumerate}

\section{Bounds via coupling}
\label{sec:coupling}

The coupling method is a powerful tool in probability theory that enables one to compare two distributions.
Numerous works have utilized the technique to obtain useful convergence bounds for a wide range of important Markov chains.
See \cite{aldous1983random,bubley1997path,lindvall1986coupling,burdzy2000efficient,pillai2017kac,durmus2019high,eberle2019quantitative,bou2020coupling}, just to name several.
In this section, we describe the general idea of this approach, and illustrate it through a few simple examples.
In particular, we use it to derive a convergence bound from a set of ``drift and minorization conditions."

\subsection{Basic theory}

For $\mu, \nu \in \mathcal{P}(\X)$, a coupling of theirs is a distribution in $\mathcal{P}(\X^2)$, say~$\gamma$, such that $\gamma(A \times \X) = \mu(A)$ and $\gamma(\X \times A) = \nu(A)$ for $A \in \B$.
In other words,~$\gamma$ is a coupling of~$\mu$ and~$\nu$ if it is the joint distribution of some random vector $(X,Y)$ such that, marginally, $X \sim \mu$ and $Y \sim \nu$.
Denote the set of all couplings of~$\mu$ and~$\nu$ by $C(\mu, \nu)$.
Suppose that we can find a measurable function $\Dist: \X^2 \to [0,\infty]$ such that, for $f \in \F$,
\begin{equation} \label{ine:lipschitz}
	|f(x) - f(y)| \leq   \Dist(x,y), \quad (x,y) \in \X^2.
\end{equation}
(The function $\Dist(\cdot,\cdot)$ is often some semi-metric.)
Then, for $\mu, \nu \in \F'$ and $\gamma \in C(\mu,\nu)$, 
\begin{equation} \label{ine:lipschitz-1}
\begin{aligned}
	\|\mu - \nu\|_{\F} &= \sup_{f \in \F} |\mu f - \nu f| \\
	&= \sup_{f \in \F} \left| \int_{\X^2} [f(x) - f(y)] \, \gamma(\df (x,y)) \right| \\
	&\leq \int_{\X^2} \Dist(x,y) \, \gamma(\df (x,y)).
\end{aligned}
\end{equation}
If one can construct a random vector $(X,Y)$ whose joint distribution is in $C(\mu,\nu)$, then $\|\mu - \nu\|_{\F} \leq E[\Dist(X,Y)]$.
In particular, if one can, on some probability space, define a copy of $X_t$ along with a random element $Y_t$ such that $X_t \sim \mu K^t$ and $Y_t \sim \varpi$, i.e., $(X_t, Y_t) \sim \gamma_t \in C(\mu K^t, \varpi)$, then $\|\mu K^t - \varpi\|_{\F} \leq E[\Dist(X_t, Y_t)]$.
Usually, to obtain a sharp bound, $X_t$ and $Y_t$ need to be correlated in some suitable manner.

This approach can be used to bound the total variation and the 1-Wasserstein distances.
Indeed, if $\F$ is the set of functions~$f$ such that $\sup_{x \in \X} |f(x)| = 1/2$, then \eqref{ine:lipschitz} holds for $f \in \F$ when $\Dist(x,y) \geq \ind_{x \neq y}$.
Thus, if $(X,Y) \sim \gamma \in C(\mu,\nu)$, then taking $D(x,y) = \ind_{x \neq y}$ yields
\begin{equation} \label{ine:coupling-tv}
\|\mu - \nu\|_{\tv} \leq \int_{\X^2} \ind_{x \neq y} \, \gamma(\df (x,y)) = P(X \neq Y).
\end{equation}
If $\F$ is the set of functions~$f$ such that $\sup_{x \neq y} |f(x) - f(y)|/\dist(x,y) = 1$, then \eqref{ine:lipschitz} holds for $f \in \F$ when $\Dist(x,y) \geq \dist(x,y)$.
Thus, if $(X,Y) \sim \gamma \in C(\mu,\nu)$, then taking $D(x,y) = \psi(x,y)$ yields
\[
W_\dist(\mu, \nu) \leq \int_{\X^2} \dist(x,y) \, \gamma(\df (x, y)) = E[\dist(X,Y)].
\]
This is obvious if one knows the more standard definition of the 1-Wasserstein distance:
\[
W_\dist(\mu, \nu) = \inf_{\gamma \in C(\mu,\nu)} \int_{\X^2} \dist(x,y) \, \gamma(\df (x, y)).
\]
It is worth noting that, in the above display, there always exists a coupling $\gamma$ that attains the infimum \citep[see, e.g.,][Theorem 4.1]{villani2008optimal}.

It is common (but not always optimal) that couplings of $\mu K^t$ and $\varpi = \varpi K^t$ are constructed in a Markovian manner.
That is, one constructs a bivariate Markov chain $(X_t, Y_t)_{t=0}^{\infty}$ with state space $\X^2$ such that the distribution of $(X_t, Y_t)$ is in $C(\mu K^t, \varpi K^t)$ for $t \in \mathbb{N}$.
This can be achieved if $X_0 \sim \mu$, $Y_0 \sim \varpi$, and the Mtk of the bivariate chain, denoted by $\tilde{K}: \X^2 \times \B^2 \to [0,1]$, is a coupling kernel of~$K$ in the following sense:
For $x, y \in \X$, $\tilde{K}((x,y), \cdot)$ is in $C(\delta_x K, \delta_y K)$, where $\delta_x$ is the point mass at $x$ (i.e., $\delta_x(A) = \ind_{x \in A}$ for $A \in \B$) so that $\delta_x K(\cdot) = K(x,\cdot)$.
In other words, given $(X_t, Y_t) = (x,y)$, $X_{t+1}$ is distributed as $K(x,\cdot)$, and $Y_{t+1}$ is distributed as $K(y,\cdot)$.

A coupling kernel always exists, since we can let
\[
\int_A \tilde{K}((x,y), \df(x',y')) = \int_A K(x, \df x') K(y, \df y'), \quad (x,y) \in \X^2, \, A \in \B^2.
\]
But this construction wouldn't be very helpful since there is no dependence between $X_t$ and $Y_t$ conditioning on $(X_0, Y_0)$, which usually renders the bound $\|\mu K^t - \varpi\|_{\F} \leq E[\Dist(X_t, Y_t)]$ too loose.
The following is an elementary result that provides a more useful coupling kernel under a simple but restrictive condition.

\begin{theorem} \label{thm:doeblin}
	Suppose that there exist $\varepsilon > 0$ and a probability measure $\nu \in \mathcal{P}(\X)$ such that, for $x \in \X$ and $A \in \B$,
	\[
	K(x,A) \geq \varepsilon \nu(A).
	\] 
	(This is called Doeblin's, or a global minorization condition.)
	Then one may construct a coupling kernel $\tilde{K}$ of $K$ such that
	\begin{equation} \label{ine:contraction-doeblin}
	\int_{\X^2} \ind_{x' \neq y'} \tilde{K}((x,y), \df(x',y')) \leq (1-\varepsilon) \ind_{x \neq y}
	\end{equation}
	for $(x,y) \in \X^2$.
\end{theorem}

\begin{remark} \label{rem:doeblin}
	If $(X_t, Y_t)_{t=0}^{\infty}$ is a bivariate chain whose Mtk satisfies \eqref{ine:contraction-doeblin}, then, for $t \in \mathbb{N}$, $P(X_{t+1} \neq Y_{t+1} \mid X_t, Y_t) = E(\ind_{X_{t+1} \neq Y_{t+1}} \mid X_t, Y_t ) \leq (1-\varepsilon) \ind_{X_t \neq Y_t}$.
\end{remark}

\begin{proof} [Proof of Theorem \ref{thm:doeblin}]
	Let $(X_t, Y_t)_{t=0}^{\infty}$ be a bivariate Markov chain with state space $\X^2$ that evolves as follows.
	Suppose that the current state is $(X_t, Y_t) = (x,y) \in \X^2$.
	If $x = y$, let $X_{t+1} = Y_{t+1}$ be distributed as $K(x, \cdot)$.
	When $x \neq y$, do the following.
	With probability $\varepsilon$, let $X_{t+1} = Y_{t+1}$ be distributed as $\nu$;
	with probability $1 - \varepsilon$ (if $\varepsilon < 1$), let $X_{t+1}$ be distributed according to the probability measure 
	\begin{equation} \label{eq:Xres}
	A \mapsto \frac{K(x,A) - \varepsilon \nu(A)}{1 - \varepsilon},
	\end{equation}
	and, independently, let $Y_{t+1}$ be distributed according to the probability measure
	\begin{equation} \label{eq:Yres}
		A \mapsto \frac{K(y,A) - \varepsilon \nu(A)}{1 - \varepsilon}.
	\end{equation}
	Note that the two measures are well-defined whenever $\varepsilon \in (0,1)$ due to Doeblin's condition.
	Evidently, given $(X_t, Y_t) = (x,y)$, $X_{t+1}$ is distributed as $K(x, \cdot)$, and $Y_{t+1}$ is distributed as $K(y, \cdot)$.
	Thus, the Mtk of the bivariate chain, which we denote by~$\tilde{K}$, is a coupling kernel of~$K$.
%	Let $\tilde{K}$ be the Mtk of the bivariate chain.
%	We claim that, for $(x,y) \in \X^2$, $\tilde{K}((x,y), \cdot)$ is a coupling of $K(x, \cdot)$ and $K(y, \cdot)$.
%	This is clear when $x = y$.
%	When $x \neq y$, we can see this by noting that, for $A \in \B$,
%	\[
%	\tilde{K}((x,y), A \times \X) = \varepsilon \nu(A) + (1- \varepsilon) \frac{K(x,A) - \varepsilon \nu(A)}{1 - \varepsilon} = K(x, A),
%	\]
%	and, analogously, $\tilde{K}((x,y), \X \times A) = K(y, A)$.
%	Thus, $\tilde{K}$ is a coupling kernel of~$K$.
	
	By construction, for $x, y \in \X^2$,
	\[
	\int_{\X^2} \ind_{x' = y'} \tilde{K}((x,y), \df(x',y'))  \geq \begin{cases}
		\varepsilon, & x \neq y, \\
		1, & x = y.
	\end{cases}
	\]
	This establishes \eqref{ine:contraction-doeblin} for $(x,y) \in \X^2$.
\end{proof}

Consider the independent Metropolis Hastings chain in Example \ref{ex:IMH}.
Since $s(x) \leq M_s$ for $x \in \X$, it holds that, for $x \in \X = [0,1]$ and $A \in \B$,
\begin{equation} \label{ine:doeblin-IMH}
K_s(x,A) \geq \int_A \inf_{x \in [0,1]} a_s(x,x') \, \df x' = \int_A \min \left\{ 1, \frac{s(x')}{M_s} \right\} \, \df x' = \int_A \frac{s(x')}{M_s} \, \df x' = \frac{1}{M_s} \pi_s(A).
\end{equation}
That is, Doeblin's condition holds with $\varepsilon = 1/M_s$.
Hence, there exists a coupling kernel of $K_s$ satisfying \eqref{ine:contraction-doeblin} for $(x,y) \in [0,1]^2$ with $1-\varepsilon = 1-1/M_s$.
On the other hand, the Gaussian chain in Example \ref{ex:gaussian} does not satisfy Doeblin's condition with any positive $\varepsilon$ since, for any bounded $A \in \B$, $\inf_{x \in \X} K_{p,\alpha}(x,A) = 0$.

%For Example \ref{ex:gaussian}, a coupling kernel can be constructed as follows:
%Let $N$ be an $\mbox{N}_p(0, I_p)$ distributed random vector.
%For $(x,y) \in \mathbb{R}^p$, let $\tilde{K}_{p,\alpha}((x,y), \cdot)$ be the joint distribution of $\alpha x + \sqrt{1-\alpha^2} N$ and $\alpha y + \sqrt{1-\alpha^2} N$.
%Then $\tilde{K}_{p,\alpha}$ is a coupling kernel of $K_{p,\alpha}$.

As implied by Remark \ref{rem:doeblin}, \eqref{ine:contraction-doeblin} is a type of ``contraction condition" that indicates $\ind_{X_t \neq Y_t}$ decreases in expectation at a geometric rate as $t$ grows.
Let us now show exactly how a coupling kernel that satisfies a contraction condition can be used to construct a convergence bound.

\begin{theorem} \label{thm:contraction}
	Let $\tilde{K}$ be a coupling kernel of~$K$.
	Suppose that there exist a constant $\rho < 1$ and a measurable function $\Dist: \X^2 \to [0,\infty]$ satisfying \eqref{ine:lipschitz} for all $f \in \F$ such that the following contraction condition holds:
	\begin{equation} \label{ine:contraction}
	\int_{\X^2} \Dist(x',y') \, \tilde{K}((x,y), \df (x', y')) \leq \rho \Dist(x,y)
	\end{equation}
	for $(x,y) \in \X^2$.
	Then, for $\mu \in \F'$ and $t \in \mathbb{N}$, 
	\[
	\|\mu K^t - \varpi\|_{\F} \leq \int_{\X^2} \Dist(x,y) \, \gamma(\df (x,y)) \, \rho^t,
	\]
	where $\gamma$ is any coupling of~$\mu$ and~$\varpi$.
\end{theorem}
\begin{proof}
	Let $\mu \in \F'$ be arbitrary.
	Let $(X_t, Y_t)_{t=0}^{\infty}$ be a bivariate chain associated with~$\tilde{K}$ such that $(X_0, Y_0) \sim \gamma \in C(\mu, \varpi)$.
	Then the distribution of $(X_t, Y_t)$ is in $C(\mu K^t, \varpi)$.
	Thus, by \eqref{ine:lipschitz-1}, for $t \in \mathbb{N}$,
	\[
	\|\mu K^t - \varpi\|_{\F} \leq E[\Dist(X_t, Y_t)].
	\]
	On the other hand, by \eqref{ine:contraction}, for $t \in \mathbb{N}$,
	\[
	E[\Dist(X_{t+1}, X_{t+1}) \mid X_t, Y_t] \leq \rho \Dist(X_t, Y_t).
	\]
	By the two displays above and the tower property of conditional expectations,
	\[
	\|\mu K^t - \varpi\|_{\F} \leq E[\Dist(X_t, Y_t)] \leq \rho^t E[\Dist(X_0, Y_0)] = \int_{\X^2} \Dist(x,y) \, \gamma(\df (x,y)) \rho^t.
	\]
\end{proof}

We may apply Theorem \ref{thm:contraction} to Example \ref{ex:IMH}.
It was already demonstrated through Theorem \ref{thm:doeblin} that there is a coupling kernel of $K_s$ that satisfies \eqref{ine:contraction} for $(x,y) \in [0,1]^2$ with $\Dist(x,y) = \ind_{x \neq y}$ and $\rho = 1 - 1/M_s$.
Let $\F$ be the set of functions~$f$ such that $\sup_{x \neq \X} |f(x)| = 1/2$ so that $\|\cdot - \cdot \|_{\F}$ corresponds to the total variation distance, and \eqref{ine:lipschitz} holds for $f \in \F$.
By Theorem \ref{thm:contraction}, for $\mu \in \mathcal{P}(\X)$,
\[
\|\mu K_s^t - \pi_s\|_{\tv} \leq  \int_{\X^2} \ind_{x \neq y} \, \mu(\df x) \, \pi_s(\df y) \left(1 - \frac{1}{M_s} \right)^t \leq \left(1 - \frac{1}{M_s} \right)^t.
\]
If $M_s$ is known, then this is a fully computable convergence bound for the independent Metropolis Hastings chain.

Turning to the Gaussian chain in Example \ref{ex:gaussian}, we note that Theorem \ref{thm:doeblin} is insufficient to provide a coupling kernel with a proper contraction condition.
To effectively utilize Theorem \ref{thm:contraction} in this context, we would need some other techniques for constructing coupling kernels.
This will be resolved in Section \ref{ssec:oneshot}.

\subsection{One-shot coupling} \label{ssec:oneshot}

To obtain a sharp convergence bound using Theorem~\ref{thm:contraction}, one needs to find a coupling kernel $\tilde{K}$  such that $\rho$ is small (or at the very least, strictly less than~1) in the contraction condition~\eqref{ine:contraction}.
This is not always easy, but for some functions $\Dist(\cdot,\cdot)$ it may be easier than for others.
Of course, the choice of $\Dist(\cdot,\cdot)$ depends on the distance function $\|\cdot - \cdot\|_{\F}$ because of the restriction~\eqref{ine:lipschitz}.
When a bound in terms of the total variation distance is desired, $\Dist(x,y)$ is often taken to be a weighted version of $\ind_{x \neq y}$, i.e., $\ind_{x \neq y}\, h(x,y)$ for some $h(x,y) \geq 1$; see Section \ref{ssec:drift}.
When a bound in terms of the 1-Wasserstein distance induced by $\dist(\cdot,\cdot)$ is desired, one may consider $\Dist(x,y)$ of the form $\dist(x,y)^r h(x,y)^{1-r}$ for some $h(x,y) \geq 1$ and $r \in (0,1]$ \citep[][Section 20.4]{douc2018markov}.
Sometimes, for the distance function $\|\cdot - \cdot\|_{\F}$ that we are interested in, it is too difficult to establish a good contraction condition for a function $\Dist(\cdot,\cdot)$ that satisfies \eqref{ine:lipschitz} for $f \in \F$.
In such cases, it may be helpful to consider some other distance function $\|\cdot - \cdot\|_{\mathcal{G}}$ first, establish a good contraction condition suitable for that distance, and then transform the resulting convergence bound to one in terms of the original distance.

Starting from a convergence bound in terms of the 1-Wasserstein distance, it is often possible to obtain a bound in terms of the total variation distance through a technique called ``one-shot coupling" \citep{roberts2002one,madras2010quantitative}.
In particular, one can use the following theorem.

\begin{theorem} \citep[][Theorem 12]{madras2010quantitative} \label{lem:continuity}
	% Suppose that $(\X,d)$ is a Polish metric space and $\B$ is its Borel algebra.
	Let $\F'$ be the set of probability measures $\mu \in \mathcal{P}(\X)$ such that $\int_{\X} \dist(x_0, x) \, \mu(\df x)$ for some $x_0 \in \X$. 
	Let $\nu$ be a $\sigma$-finite measure on $(\X,\B)$.
	Assume that there is a measurable function $k: \X^2 \to [0,\infty)$ such that, for $x \in \X$ and $A \in \B$, $K(x,A) = \int_A k(x,x') \, \nu(\df x')$.
	Suppose further that there exists a constant $b < \infty$ such that
	\begin{equation} \label{ine:continuity}
		1 -  \int_{\X} \min \{ k(x,x'), k(y,x')\} \, \nu(\df x') \leq b \dist(x,y)
	\end{equation}
	for $x,y \in \X$.
	Then, for $t \in \mathbb{N}$ and $\mu \in \F'$,
	\[
	\|\mu K^{t+1} - \varpi\|_{\tv} \leq b W_\dist(\mu K^t, \varpi).
	\]
\end{theorem}

\begin{proof}
	Recall that
	\[
	W_\dist(\mu K^t, \varpi) = \inf_{\gamma_t \in C(\mu K^t, \varpi)} \int_{\X^2} \dist(x,y) \, \gamma(\df (x,y)),
	\]
	and one can find $\gamma_t \in C(\mu K^t, \varpi)$ that attains the infimum.
	Let $(X,Y) \in \gamma_t$, so that
	\begin{equation} \label{eq:madrascoupling}
		W_\dist(\mu K^t, \varpi) = E [\dist(X,Y)]
	\end{equation}
	
	For $x, y \in \X$, let 
	\[
	q_{x,y}(z) = \min \{ k(x,z), k(y,z) \}, 
	\]
	and set $a_{x,y} = \int_{\X} q_{x,y}(z) \, \nu(\df z)$, so that $1-a_{x,y}$ is the left-hand-side of \eqref{ine:continuity}.
	Given $(X,Y) = (x,y)$, generate $(X',Y')$ in the following manner:
	With probability $a_{x,y}$, let $X' = Y'$ be distributed according to the probability density function $q_{x,y}(\cdot)/a_{x,y}$;
	with probability $1-a_{x,y}$ (if $a_{x,y} < 1$), let $X'$ be distributed according to the probability measure
	\[
	A \mapsto \frac{K(x,A) - \int_A q_{x,y}(z) \, \nu(\df z) }{1- a_{x,y}}
	\]
	and, independently, let $Y'$ be distributed according to the probability measure
	\[
	A \mapsto \frac{K(y,A) - \int_A q_{x,y}(z) \, \nu(\df z) }{1- a_{x,y}}.
	\]
	It is easy to see that, given $(X,Y) = (x,y)$, $X'$ is distributed as $K(x, \cdot)$ and $Y'$ is distributed as $K(y, \cdot)$.
	Thus, marginally,~$X'$ is distributed as
	\[
	\int_{\X^2} K(x, \cdot) \, \gamma_t(\df (x, y)) = \int_{\X} \mu K^t(\df x) K(x, \cdot)  = \mu K^{t+1}(\cdot),
	\]
	while $Y'$ is distributed as $\varpi K(\cdot) = \varpi(\cdot)$.
	Hence, the joint distribution of $(X', Y')$ is in $C(\mu K^{t+1}, \varpi)$.

	By~\eqref{ine:coupling-tv}, \eqref{ine:continuity}, and \eqref{eq:madrascoupling},
	\[
	\begin{aligned}
		\|\mu K^{t+1} - \varpi\|_{\tv} &\leq P(X' \neq Y') \\
		&= E[ P( X' \neq Y'  \mid X, Y) ] \\
		&\leq E(1-a_{X,Y}) \\
		&\leq bE[\dist(X,Y)] \\
		&= b W_\dist(\mu K^t, \varpi).
	\end{aligned}
	\]
\end{proof}

\subsubsection{Application to the Gaussian chain}

Let us apply Theorem \ref{thm:contraction} and Theorem \ref{lem:continuity} to the Gaussian chain in Example \ref{ex:gaussian}, and use it as a stage for discussing some practical issues concerning the application of coupling methods.

Suppose that our goal is to bound $\|\mu K_{p,\alpha}^t - \varpi_p\|_{\tv}$ from above.
To apply Theorem~\ref{thm:contraction}, one could let $\Dist(x,y) = \ind_{x \neq y}$ for $x, y \in \mathbb{R}^p$.
But, in this case, it is impossible to establish a nontrivial contraction condition.
Indeed, by~\eqref{ine:coupling-tv}, for any coupling kernel $\tilde{K}$ of $K_{p,\alpha}$ and $x,y \in \mathbb{R}^p$,
\[
\int_{\mathbb{R}^p \times \mathbb{R}^p} \ind_{x' \neq y'} \, \tilde{K}((x,y), \df (x', y')) \geq \|\delta_x K_{p,\alpha} - \delta_y K_{p,\alpha}\|_{\tv}.
\]
But, by~\eqref{eq:tv}, whenever $\alpha > 0$, the total variation distance between $\mbox{N}_p(\alpha x, (1-\alpha^2) I_p )$ and $\mbox{N}_p(\alpha y, (1-\alpha^2) I_p )$ goes to 1 as $\|x - y\| \to \infty$.
Thus, the contraction condition
\[
\int_{\mathbb{R}^p \times \mathbb{R}^p} \ind_{x' \neq y'} \, \tilde{K}((x,y), \df (x', y')) \leq \rho \ind_{x \neq y}, \quad x, y \in \mathbb{R}^p,
\]
can only hold when $\rho \geq 1$.
One could get somewhere if $\Dist(x,y)$ is taken to be, say, $\ind_{x \neq y} \, [p^{-1} \|x\|^2 + p^{-1} \|y\|^2 + 1]^{1-r}$ for some $r \in (0,1)$; see Section~\ref{ssec:drift}.
However, it turns out that it is much easier to construct a sharp convergence bound by first considering the 1-Wasserstein distance induced by the Euclidean distance, and then utilize one-shot coupling.

For $x,y \in \mathbb{R}^p$, let $\dist(x,y) = \|x - y\|$.
Let $\F'_p$ be the set of probability measures~$\mu$ such that $\int_{\mathbb{R}^p} \|x\| \, \mu(\df x) < \infty$.
We will bound $W_\dist(\mu K_{p,\alpha}^t, \varpi_p)$ from above for $\mu \in \F'_p$ using Theorem~\ref{thm:contraction}.

There are many possible ways of constructing coupling kernels.
Here, we use a construction sometimes referred to as the ``common random number coupling."
Let $N$ be an $\mbox{N}_p(0, I_p)$ distributed random vector.
For $(x,y) \in \mathbb{R}^p$, let $\tilde{K}_{p,\alpha}((x,y), \cdot)$ be the joint distribution of $\alpha x + \sqrt{1-\alpha^2} N$ and $\alpha y + \sqrt{1-\alpha^2} N$.
Since $\alpha x' + \sqrt{1-\alpha^2} N \sim K_{p,\alpha}(x', \cdot)$ for $x' \in \mathbb{R}^p$, $\tilde{K}_{p,\alpha}$ is a coupling kernel of $K_{p,\alpha}$. 
For $(x,y) \in \mathbb{R}^p \times \mathbb{R}^p$, 
\[
\begin{aligned}
	\int_{\mathbb{R}^p \times \mathbb{R}^p} \|x'-y'\| \, \tilde{K}_{p,\alpha}((x,y), \df (x',y')) &= E [\| (\alpha x + \sqrt{1-\alpha^2}N) - (\alpha y + \sqrt{1-\alpha^2}N) \| ] \\
	&= \alpha \|x-y\|.
\end{aligned}
\]
Thus, the contraction condition \eqref{ine:contraction} holds on $\mathbb{R}^p \times \mathbb{R}^p$ for $\tilde{K} = \tilde{K}_{p,\alpha}$ with $\Dist(\cdot,\cdot) = \dist(\cdot, \cdot)$ and $\rho = \alpha$.
By Theorem~\ref{thm:contraction} and the triangle inequality, for $t \in \mathbb{N}$,
\begin{equation} \label{ine:Wasserstein-Gaussian}
W_\dist(\mu K_{p,\alpha}^t, \varpi_p) \leq \left( \int_{\mathbb{R}^p} \|x\| \, \mu(\df x) +  \int_{\mathbb{R}^p} \|y\| \, \varpi_p(\df y) \right) \alpha^t.
\end{equation}
This bound reveals that $\mu K_{p,\alpha}^t$ approaches $\varpi_p$ in the 1-Wasserstein distance at a geometric rate~$\alpha$.

If this were a practical problem, the stationary distribution $\varpi_p$ would likely be intractable, and one would not be able to evaluate $\int_{\mathbb{R}^p} \|y\| \, \varpi_p(\df y)$.
It is however possible to bound the integral from above via the following result for a generic Markov chain with Mtk $K(\cdot,\cdot)$.

\begin{theorem}\citep[][Proposition 4.24]{hairer2006ergodic} \label{lem:piv}
	Let $h: \X \to [0,\infty)$ be a measurable function.
	Suppose that there exist $\lambda \in [0,1)$ and $L \in [0,\infty)$ such that, for $x \in \X$,
	\[
	\int_{\X} h(x') K(x, \df x') \leq \lambda h(x) + L.
	\]
	(This is called a drift condition.)
	Then 
	\[
	\int_{\X} h(x) \, \varpi(\df x) \leq \frac{L}{1 - \lambda}.
	\]
\end{theorem}
\begin{proof}
	By the drift condition, for $x \in \X$ and $t \in \mathbb{N}_+$,
	\[
	\int_{\X} h(x') K^t(x, \df x') \leq \lambda^t h(x) + \frac{L(1-\lambda^t)}{1 - \lambda}.
	\]
	For $n \in \mathbb{N}_+$, let $h_n: \X \to [0,\infty)$ be such that $h_n(x) = \min \{h(x), n\}$.
	Then
	\[
	\int_{\X} h_n(x') K^t(x, \df x') \leq \min \left\{ \lambda^t h(x) + \frac{L(1-\lambda^t)}{1 - \lambda}, n \right\}.
	\]
	Integrating with respect to $\varpi$ yields
	\[
	\int_{\X} h_n(x) \, \varpi(\df x) = \int_{\X^2} h_n(x') K^t(x, \df x') \, \varpi(\df x) \leq \int_{\X} \min \left\{ \lambda^t h(x) + \frac{L(1-\lambda^t)}{1 - \lambda}, n \right\} \varpi(\df x).
	\]
	By the dominated convergence theorem, letting $t \to \infty$ shows that
	\[
	\int_{\X} h_n(x) \, \varpi(\df x) \leq \min\left\{ \frac{L}{1 - \lambda} , n \right\}.
	\]
	By the monotone convergence theorem, letting $n \to \infty$ yields the desired result.
\end{proof}

In Example \ref{ex:gaussian}, we can pretend that we know little of $\varpi_p$, and bound $\int_{\mathbb{R}^p} \|y\| \, \varpi_p(\df y)$ in the following manner.
Note that
\[
\int_{\mathbb{R}^p} \|x'\|^2 \, K_{p,\alpha}(x, \df x') = \alpha^2 \|x\|^2 + (1-\alpha^2) p.
\]
Thus, by Theorem~\ref{lem:piv},
\begin{equation} \label{ine:int-xnorm}
\int_{\mathbb{R}^p} \|x\| \, \varpi_p(\df y) \leq \sqrt{ \int_{\mathbb{R}^p} \|x\|^2 \, \varpi_p(\df y) } \leq \sqrt{\frac{(1-\alpha^2)p}{1-\alpha^2}} = \sqrt{p}.
\end{equation}

Let us now apply Theorem~\ref{lem:continuity} to transform the Wasserstein distance bound into a total variation bound.
For $x \in \mathbb{R}^p$, let $k_{p,\alpha}(x,\cdot)$ be the density function of the $\mbox{N}_p(\alpha x, (1-\alpha^2)I_p)$ distribution, i.e., $K_{p,\alpha}(x,\cdot)$, with respect to the Lebesgue measure.
A bit of calculus reveals that, for $x,y \in \mathbb{R}^p$,
\begin{equation} \label{ine:normaltv}
1 - \int_{\mathbb{R}^p} \min \{ k_{p,\alpha}(x, x'), k_{p,\alpha}(y, x') \} \, \df x' = 1 - 2 F_N \left( - \frac{\alpha \|x-y\|}{2\sqrt{1-\alpha^2}} \right) \leq \frac{\alpha}{\sqrt{(2\pi) (1-\alpha^2) }} \|x-y\|,
\end{equation}
where $F_N(\cdot)$ is the cumulative distribution function of the one-dimensional standard normal distribution.
By Theorem~\ref{lem:continuity} along with \eqref{ine:Wasserstein-Gaussian} to \eqref{ine:normaltv}, for $\mu \in \F_p'$ and $t \in \mathbb{N}$,
\begin{equation} \label{ine:Gaussian-wasser-tv}
\| \mu K_{p,\alpha}^{t+1} - \varpi_p \|_{\tv} \leq \frac{\alpha}{\sqrt{(2\pi) (1-\alpha^2) }} \left( \int_{\mathbb{R}^p} \|x\| \, \mu(\df x) +  \sqrt{p} \right) \alpha^t.
\end{equation}

This bound indicates that $\mu K_{p,\alpha}^t$ approaches $\varpi_p$ in the total variation distance at a geometric rate $\alpha$, which does not depend on the dimension $p$.
By considering the $L^2$ convergence rate of the chain (see Section \ref{ssec:l2basic}) and utilizing Theorem 2.1 of \cite{roberts1997geometric}, it is possible to show that the convergence rate indicated by \eqref{ine:Gaussian-wasser-tv} is in fact sharp.
That is, it is not possible to find a bound of the form $\| \mu K_{p,\alpha}^t - \varpi_p \|_{\tv} \leq C_{p,\alpha,\mu} \rho_{p,\alpha}^t$, where $C_{p,\alpha,\mu} < \infty$ and $\rho_{p,\alpha} < \alpha$, that holds for all $\mu \in \F_p'$ and $t \in \mathbb{N}_+$.
This is an example of the convergence bound scaling well with the dimension in the sense that the convergence rate indicated by the bound does not deteriorate (go to 1) faster than the true convergence rate does as the dimension grows.
Good scaling is not always easily achieved, as obtaining sharp convergence bounds often becomes more difficult in problems with higher dimensions.
In Section \ref{ssec:drift}, we give a type of bound that is tremendously popular and powerful in certain settings but scales poorly with the dimension when applied to Example \ref{ex:gaussian}.

We may also use \eqref{ine:Gaussian-wasser-tv} to obtain a bound on the mixing time.
For $\mu \in \F_p'$ and $\epsilon > 0$, let $t_{\tv}(\epsilon, \mu)$ be the smallest $t \in \mathbb{N}_+$ such that $\|\mu K^t - \varpi \|_{\tv} \leq \epsilon$.
Then \eqref{ine:Gaussian-wasser-tv} implies that
\[
t_{\tv}(\epsilon, \mu) \leq \left\lceil (-\log \alpha)^{-1} \left\{ \log \left[ \frac{\alpha}{\sqrt{(2\pi) (1-\alpha^2) }} \left( \int_{\mathbb{R}^p} \|x\| \, \mu(\df x) +  \sqrt{p} \right) \right] - \log \epsilon \right\} \right\rceil
\]
if $\alpha \in (0,1)$.
If $ \int_{\mathbb{R}^p} \|x\| \, \mu(\df x) = O(p)$ as $p \to \infty$, then $t_{\tv}(\epsilon, \mu) = O(\log p)$.

\subsection{Drift and minorization} \label{ssec:drift}

As a final application of the coupling method, we use it to derive a convergence bound based on a set of drift and minorization conditions.
Many important convergence bounds in the literature are constructed based on this type of condition \citep{tierney1994markov,rosenthal1995minorization,mengersen1996rates,roberts1999bounds,meyn2012markov}.
While these bounds are often far from sharp \citep{qin2020limitations}, they are very powerful for establishing qualitative results like geometric ergodicity \citep{hobert1998geometric,jarner2000geometric,jones2001honest,jones2004sufficient,roy2007convergence,khare2013geometric,livingstone2019geometric}.

Recall the drift condition in Theorem~\ref{lem:piv}.
We say that the Mtk~$K$ satisfies a drift condition with drift function $h: \X\to [0,\infty)$ if there exist $\lambda \in [0,1)$ and $L \in [0,\infty)$ such that, for $x \in \X$,
\[
Kh(x) := \int_{\X} h(x') K(x, \df x') \leq \lambda h(x) + L.
\]
We say that~$K$ satisfies a minorization condition associated with~$h$ if there exist $\varepsilon > 0$, a probability measure $\nu \in \mathcal{P}(\X)$, and $\Delta > 2L/(1-\lambda)$ such that
\[
K(x, A) \geq \varepsilon \nu(A)
\]
for $A \in \B$ whenever $h(x) \leq \Delta$.

The following convergence bound is reminiscent of a famous result from \cite{rosenthal1995minorization}.
Its proof uses ideas from \cite{hairer2011yet,hairer2011asymptotic,butkovsky2014subgeometric,douc2018markov}.

\begin{theorem} \label{thm:drift}
	Suppose that the above drift and minorization conditions hold.
	Then, for $\mu \in \mathcal{P}(\X)$, $r \in (0,1)$, and $t \in \mathbb{N}$,
	\[
	\| \mu K^t - \varpi \|_{\tv} \leq \left( \mu h + \frac{L}{1 - \lambda} + 1 \right) \rho^t.
	\]
	where
	\begin{equation} \label{eq:rho}
		\rho = \max\left\{(1 - \varepsilon)^r (2L+1)^{1-r},  \left( \lambda + \frac{2L+1-\lambda}{\Delta+1} \right)^{1-r} \right\}.
	\end{equation}
\end{theorem}

\begin{remark}
	Note that $\Delta > 2L/(1-\lambda)$ implies that $\lambda +  (2L+1-\lambda)/(\Delta+1) < 1$.
	Thus, there exists $r \in (0,1)$ such that $\rho$, as given in \eqref{eq:rho}, is strictly less than~1.
\end{remark}

\begin{proof} [Proof of Theorem~\ref{thm:drift}]
	We will make use of Theorem~\ref{thm:contraction}.
	Take $\F$ to be the set of functions $f$ such that $\sup_{x \in \X} |f(x)| = 1/2$ and $\F' = \mathcal{P}(\X)$, so that $\|\cdot - \cdot\|_{\F} = \|\cdot - \cdot\|_{\tv}$.
	
	Fix $r \in (0,1)$.
	For $(x,y) \in \X^2$, let 
	\[
	\Dist(x,y) = \ind_{x \neq y} [h(x) + h(y) + 1]^{1-r}.
	\]
	Then $\Dist(x,y) \geq \ind_{x \neq y} \geq |f(x) - f(y)|$ for $f \in \F$, i.e., \eqref{ine:lipschitz} holds for $f \in \F$.
	
	Let us now construct a coupling kernel that satisfies a contraction condition.
	The construction is somewhat similar to the one in the proof of Theorem \ref{thm:doeblin}.
	Let $S = \{x \in \X: \, h(x) \leq \Delta \}$.
	Let $(X_t, Y_t)_{t=0}^{\infty}$ be a bivariate Markov chain that evolves as follows.
	Suppose that the current state is $(X_t, Y_t) = (x,y)$.
	When $x = y$, simply generate $X_{t+1} = Y_{t+1}$ according to the probability measure $K(x, \cdot)$.
	When $x \neq y$, do the following.
	If $(x,y) \not\in S^2$, then generate $X_{t+1} \sim K(x, \cdot)$ and $Y_{t+1} \sim K(y, \cdot)$ independently.
	If $(x,y) \in S^2$, then, with probability~$\varepsilon$, let $X_{t+1} = Y_{t+1}$ be distributed as~$\nu$;
	with probability $1 - \varepsilon$, let $X_{t+1}$ be distributed according to the probability measure
	\eqref{eq:Xres}, and, independently, let $Y_{t+1}$ be distributed according to the probability measure \eqref{eq:Yres}.
%	\[
%	A \mapsto \frac{K(x,A) - \varepsilon \nu(A)}{1 - \varepsilon},
%	\]
%	and, independently, let $Y_{t+1}$ be distributed according to the probability measure
%	\[
%	A \mapsto \frac{K(y,A) - \varepsilon \nu(A)}{1 - \varepsilon}.
%	\]
	It is straightforward to check that the Mtk of this chain, denoted by $\tilde{K}$, is a coupling kernel of~$K$.
	
	Now, for $(x,y) \in \X^2$, by H\"{o}der's inequality,
	\[
	\begin{aligned}
		&\int_{\X^2} \Dist(x',y') \, \tilde{K}((x,y), \df (x',y')) \\
		=& \int_{\X^2} \ind_{x' \neq y'}^r \, [h(x') + h(y') + 1]^{1-r} \, \tilde{K}((x,y), \df (x',y')) \\
		\leq & \left[ \int_{\X^2} \ind_{x' \neq y'} \, \tilde{K}((x,y), \df (x',y')) \right]^r \left\{ \int_{\X^2} [h(x') + h(y') + 1] \, \tilde{K}((x,y), \df (x',y')) \right\}^{1-r} \\
		=& \left[ \int_{\X^2} \ind_{x' \neq y'} \, \tilde{K}((x,y), \df (x',y')) \right]^r  [Kh(x) + Kh(y) + 1] ^{1-r}.
%		=& E\left\{ \ind_{X_1 \neq X_1}^r [ h(X_1) + h(Y_1) + 1 ]^{1-r} \mid (X_0, Y_0) = (x,y) \right\} \\
%		\leq& \{ E[\ind_{X_1 \neq Y_1} \mid (X_0, Y_0) = (x,y)  ] \}^r \{ E[h(X_1) + h(Y_1) + 1 \mid (X_0, Y_0) = (x,y)  ] \}^{1-r} \\
%		=& P(X_1 \neq Y_1 \mid (X_0, Y_0) = (x,y) )^r \, [Kh(x) + Kh(y) + 1]^{1-r}. 
	\end{aligned}
	\]
	By construction, 
	\[
	 \int_{\X^2} \ind_{x' \neq y'} \, \tilde{K}((x,y), \df (x',y')) \leq \begin{cases}
	 	(1 - \varepsilon) \ind_{x \neq y}, & (x,y) \in S^2, \\
	 	\ind_{x \neq y}, & \text{otherwise}.
%	P(X_1 \neq Y_1 \mid (X_0, Y_0) = (x,y) ) \leq \begin{cases}
%		1 - \varepsilon, & (x,y) \in S^2, \\
%		1, & \text{otherwise}.
	\end{cases}
	\]
	By the drift condition,
	\[
	\begin{aligned}
		Kh(x) + Kh(y) + 1 &\leq \lambda h(x) + \lambda h(y) + 2L + 1 \\
		&= \left[ \lambda + \frac{2L + 1 - \lambda}{h(x) + h(y) + 1} \right] [h(x) + h(y) + 1] \\
		&\leq \begin{cases}
			(2L+1) [h(x) + h(y) + 1], & (x,y) \in S^2, \\
			\left( \lambda + \frac{2L+1-\lambda}{\Delta+1} \right) [h(x) + h(y) + 1], & \text{otherwise}.
		\end{cases}
	\end{aligned}
	\]
	Combining terms, we find that
	\[
	\int_{\X^2} \Dist(x',y') \, \tilde{K}((x,y), \df (x',y')) \leq \rho \Dist(x,y),
	\]
	where $\rho$ is given in \eqref{eq:rho}.
	
	Let $\gamma$ be a coupling of $\mu$ and $\varpi$.
	By Theorem~\ref{thm:contraction}, for $\mu \in \mathcal{P}(\X)$ and $t \in \mathbb{N}$,
	\[
	\begin{aligned}
		\|\mu K^t - \varpi\|_{\tv} &\leq \int_{\X^2} \Dist(x,y) \, \gamma(\df (x, y)) \rho^t \\
		&\leq \int_{\X^2} [h(x) + h(y) + 1] \, \gamma(\df (x, y)) \rho^t \\
		&= (\mu h + \varpi h + 1) \rho^t.
	\end{aligned}
	\]
	Applying Theorem~\ref{lem:piv} to bound $\varpi h$ gives us the desired result.
\end{proof}

For alternative derivations of drift and minorization-based convergence bounds that rely less on the coupling method, see  \cite{meyn1994computable,baxendale2005renewal,jerison2019quantitative}.

The drift and minorization condition presented here is just one among several commonly used forms.
For some other useful versions of drift and minorization, including those used for establishing subgeometric convergence, see \cite{jarner2002polynomial,douc2004practical,douc2008bounds,butkovsky2014subgeometric,andrieu2015quantitative,durmus2016subgeometric,zhou2022dimension}.
%For some other useful forms of drift and minorization, see \cite{qin2019geometric,zhou2022dimension}.

\subsubsection{Application to the Gaussian chain}

We now apply Theorem~\ref{thm:drift} to the Gaussian chain in Example \ref{ex:gaussian}.
A drift function we can use is $h(x) = p^{-1} \|x\|^2$, $x \in \mathbb{R}^p$.
Then
\[
K_{p,\alpha} h(x) = \alpha^2 h(x) + 1-\alpha^2 .
\]
So a drift condition holds with $\lambda = \alpha^2$ and $L = 1-\alpha^2$.
Let $\Delta > 2L/(1-\lambda) = 2$, and let $S = \{x \in \mathbb{R}^p: \, h(x) \leq \Delta\}$.
Recall that $k_{p,\alpha}(x, \cdot)$ is the density of $K_{p,\alpha}(x,\cdot)$ for $x \in \X$.
For $x' \in \mathbb{R}^p$, let
\[
q(x') = \inf_{x \in S} k_{p,\alpha}(x,x') = \frac{1}{[2\pi(1-\alpha^2)]^{p/2}} \exp \left[ - \frac{(\|x'\| + \alpha\sqrt{p \Delta} )^2}{2(1-\alpha^2)} \right].
\]
Then, for $x \in S$ and $x' \in \mathbb{R}^p$,
\[
k_{p,\alpha}(x, x') \geq \varepsilon \frac{q(x')}{\int_{\mathbb{R}^p} q(x'') \, \df x'' },
\]
where 
\[
\varepsilon = \int_{\mathbb{R}^p} q(x'') \, \df x'' = \int_{\mathbb{R}^p} \frac{1}{[2\pi(1-\alpha^2)]^{p/2}} \exp \left[ - \frac{(\|x''\| + \alpha\sqrt{p\Delta} )^2}{2(1-\alpha^2)} \right] \df x''.
\]
Thus, a minorization condition holds with this value of $\varepsilon$.
Applying Theorem~\ref{thm:drift} shows that, for $\mu \in \mathcal{P}(\mathbb{R}^p)$ and $t \in \mathbb{N}$,
\[
\|\mu K_{p,\alpha}^t - \varpi_p \|_{\tv} \leq \left( \mu h + 2 \right) \rho^t,
\]
where, for some $r \in (0,1)$,
\begin{equation} \label{eq:rho-gaussian-drift}
\rho = \max\left\{ (1 - \varepsilon)^r (3 - 2\alpha^2)^{1-r}  , \left( \alpha^2 + \frac{3(1-\alpha^2)}{\Delta+1} \right)^{1-r} \right\}.
\end{equation}

For concreteness, take $p = 10$ and $\alpha = 1/2$.
Then
\[
\begin{aligned}
	\varepsilon &= \int_{\mathbb{R}^{10}} \left( \frac{2}{3 \pi} \right)^5 \exp \left[ - \frac{2 (\|x''\| + \sqrt{2.5\Delta})^2 }{3} \right] \df x'' \\
	&= \int_0^{\infty} \frac{2^6 u^9}{3^5 \times 4!} \exp \left[ - \frac{2 (u + \sqrt{2.5\Delta})^2 }{3} \right] \df u.
\end{aligned}
\]
For instance, if $\Delta = 4$, then $\varepsilon \approx 2.28 \times 10^{-7}$.
% In fact, since $\Delta > 2$, $\varepsilon > 5.04 \times 10^{-5}$.
In \eqref{eq:rho-gaussian-drift}, we can optimize the value of~$r$ to find the smallest value of~$\rho$, which yields $\rho \approx 1 - 6 \times 10^{-8}$.
Other choices of $\Delta \in (2, \infty)$ would result in values of $\rho$ that are very close to unity as well. 
When $t$ is large, the resulting upper bound on $\|\mu K_{p,\alpha}^t - \varpi_p \|_{\tv}$, which is proportional to $\rho^t$, is extremely conservative.
Indeed, recall that the much sharper bound from~\eqref{ine:Gaussian-wasser-tv} is proportional to $\alpha^t = 0.5^t$.
Through experiments one can find that the conservativeness of \eqref{eq:rho-gaussian-drift} is exacerbated when the dimension $p$ is increased.
This is consistent with empirical evidence and theoretical analyses in the existing literature, which suggest that bounds based on drift and minorization conditions typically scale poorly with dimensions \citep{qin2020limitations}.
Oftentimes, this type of bound is more suitable for establishing qualitative results such as geometric ergodicity.

\section{$L^2$ theory}
\label{sec:l2}

The $L^2$ theory for Markov chains is a framework for studying the convergence properties of a Markov chain, usually reversible, in terms of the $L^2$ distance by examining the linear operator associated with the chain's transition kernel.
A substantial body of literature works within this theoretical framework, offering a diverse array of analytical techniques \citep[][to name some]{amit1996convergence,liu:wong:kong:1994,roberts1997geometric,roberts1997updating,diaconis2000analysis,diaconis2008gibbs,hobert2008theoretical,khare2011spectral,dwivedi2019log,andrieu2022comparison}. 
We will first review some basic concepts.
Then, we explain how isoperimetric inequalities, a type of inequality that regulates the geometric features of the target distribution, can be leveraged to analyze Markov chains within this framework.
Finally, we review some simple techniques for showing how slow a chain converges.

\subsection{Basic theory} \label{ssec:l2basic}

Throughout Section~\ref{sec:l2}, let $\F$ be the set of functions~$f$ such that $\int_{\X} f(x)^2 \, \varpi(\df x) = 1$, and let $\F'$ be the set of probability measures~$\mu$ such that $\df \mu/ \df \varpi$ is squared integrable with respect to $\varpi$.
Then $\|\mu - \nu\|_{\F}$ is the $L^2$ distance $\|\mu - \nu\|_2$ for $\mu, \nu \in \F'$.

The $L^2$ theory for Markov chains begins with the examination of a linear space formed by some functions on~$\X$.
Denote by $L^2(\varpi)$ the set of real measurable functions~$f$ on $\X$ such that $\int_{\X} f(x)^2 \, \varpi(\df x) < \infty$,
with the understanding that two functions are equal if their difference is $\varpi$-almost everywhere vanishing.
For $c \in \mathbb{R}$ and $f, g\in L^2(\varpi)$, let $(-f)(x) = -f(x)$, $(cf)(x) = c f(x)$, $(f+g)(x) = f(x) + g(x)$, and $(f-g)(x) = f(x) - g(x)$ for $x \in \X$.
%Denote by $0$ be the function that is constantly zero.
Then $L^2(\varpi)$ forms a real linear space.
For $f, g \in L^2(\varpi)$, define their inner product as
\[
\langle f, g \rangle = \int_{\X} f(x) g(x) \, \varpi(\df x),
\]
and let $\|f\|_2 = \langle f, f \rangle^{1/2}$.
Then $\|\cdot\|_2$ is a norm, and shall be referred to as the $L^2(\varpi)$ norm.
It can be shown that $L^2(\varpi)$ is a Hilbert space \citep[see, e.g.,][Theorem 13.15]{bruckner2008}.

	The space $L^2(\varpi)$ provides a natural stage for studying the $L^2$ distance between distributions.
	Indeed, a distribution $\mu$ is in $\F'$ if and only if the function $\df \mu/ \df \varpi$ exists and is in $L^2(\varpi)$.
	Moreover, using the Cauchy-Schwarz inequality, one can derive \eqref{eq:l2-distance}, which states that, for $\mu, \nu \in \F'$, 
	\[
	\|\mu - \nu\|_2 = \left\| \frac{\df \mu}{\df \varpi} - \frac{\df \nu}{\df \varpi} \right\|_2.
	\]

It will be convenient to work with the subspace $L_0^2(\varpi)$, which consists of functions $f \in L^2(\varpi)$ such that $\varpi f = 0$.

\begin{remark}
	If a random element $X$ is distributed as $\varpi$, and $f$ and $g$ are in $L_0^2(\varpi)$, then $E[f(X)] = E[g(X)] = 0$, $\langle f, g \rangle = \mbox{cov} \, [f(X), g(X)]$, and $\|f\|_2^2 = \mbox{var} \, [f(X)]$.
\end{remark}

%Let $\mathcal{N}$ denote the set of measurable functions that are $\varpi$-almost everywhere zero.
%Given a measurable function $f:\X \to \mathbb{R}$, let $[f] = \{f + g: g \in \mathcal{N}\}$ be the class of functions that are equivalent to~$f$ in the sense that they are $\varpi$-almost surely equal to~$f$.
%Define the space of equivalence classes
%\[
%L_0^2(\varpi) = \left\{ [f]: \, \int_{\X} f(x)^2 \, \varpi(\df x) < \infty, \, \int_{\X} f(x) \, \varpi(\df x) = 0 \right\}.
%\]
%For $c \in \mathbb{R}$ and $[f], [g] \in L_0^2(\varpi)$, let $-[f] = [-f]$, $c [f] = [cf]$, $[f] + [g] = [f+g]$, and $[f] - [g] = [f-g]$.
%(These operations are well-defined since, e.g., $[f'+g'] = [f+g]$ so long as $f' \in [f]$ and $g' \in [g]$.)
%Then $L_0^2(\varpi)$ forms a linear space whose zero vector is $[\zero]$.
%We may further define an inner-product and a norm on $L_0^2(\varpi)$: for $[f], [g] \in L_0^2(\varpi)$, let
%\[
%\langle [f], [g] \rangle = \int_{\X} f(x) g(x) \, \varpi(\df x),
%\]
%and let $\|[f]\|_2 = \sqrt{\langle [f], [f] \rangle}$.
%It is well-known that $L_0^2(\varpi), \langle \cdot, \cdot \rangle)$ forms a Hilbert space.
%From here on, we will be lazy with our notations, and represent $[f]$ as an arbitrary function $g \in [f]$, with the understanding that two functions $f$ and $g$ are considered equal in $L^2(\varpi)$ if $[f] = [g]$, or equivalently, $\|f - g\|_2 := \|[f]-[g]\|_2 = 0$.

The Mtk $K(\cdot,\cdot)$, which satisfies $\varpi K(\cdot) = \varpi(\cdot)$, can be regarded as a bounded linear operator (called a Markov operator) on $L_0^2(\varpi)$ in the following way.
For $f \in L_0^2(\varpi)$, let
\[
Kf (x) = \int_{\X} f(x') K(x, \df x'), \quad x \in \X.
\]
%If $(X_t)_{t=0}^{\infty}$ is a chain associated with~$K$, then $Kf(x)$ is the conditional expectation $E[f(X_{t+1}) \mid X_t = x]$.
The map $f \mapsto Kf$ is clearly linear.
To see that its range is in $L_0^2(\varpi)$, note that when $f \in L_0^2(\varpi)$,
\[
\int_{\X} Kf(x) \, \varpi(\df x) = \int_{\X} \int_{\X} f(x') K(x, \df x') \, \varpi(\df x) = \int_{\X} f(x) \, \varpi(\df x) = 0,
\]
and, by Jensen's inequality,
\begin{equation} \label{ine:Jensen}
	\int_{\X} \left[ \int_{\X} K(x, \df x') f(x') \right]^2 \varpi(\df x) \leq \int_{\X} \int_{\X} f(x')^2 K(x, \df x') \, \varpi(\df x) = \int_{\X} f(x')^2 \, \varpi(\df x') < \infty.
\end{equation}
The operator norm of $K$ is
\[
\|K\|_2 = \sup_{f \in L_0^2(\varpi), \, f \neq 0} \frac{\|Kf\|_2}{\|f\|_2} = \sup_{f \in L_0^2(\varpi), \, \|f\|_2 = 1} \|Kf\|_2 .
\]
By \eqref{ine:Jensen}, $\|K\|_2 \leq 1$.

\begin{remark} \label{rem:covariance}
	Let $(X_t)_{t=0}^{\infty}$ be a chain associated with~$K$.
	Then, for $f \in L_0^2(\varpi)$, $Kf(X_0) = E[f(X_1) \mid X_0]$.
	If, furthermore, $X_0 \sim \varpi$, then, for $f, g \in L_0^2(\varpi)$,
	\[
	\|Kf\|_2^2 = \mbox{var} \, \{ E[f(X_1) \mid X_0] \}, \quad  \langle f, Kg \rangle = \mbox{cov} \, (f(X_0), g(X_1)).
	\]
\end{remark}

For $t \in \mathbb{N}_+$, the $t$-step Mtk $K^t(\cdot, \cdot)$ also defines an operator $K_t$ on $L_0^2(\varpi)$, with $K_1 f = K f$, and
\[
K_t f(x) = \int_{\X} f(x') K^t(x, \df x') = \int_{\X} Kf(x') K_{t-1}(x, \df x') 
\]
when $t \geq 2$.
Note that $K_t f$ is $K$ applied to~$f$ $t$ times.
In other words, $K_t$ is just $K^t$, the product (composition) of $t$ $K$'s.
By convention, $K^0$ is the identity operator.

The convergence behavior of a Markov chain associated with the Mtk $K(\cdot,\cdot)$ is tied to the properties of the operator~$K$.
Indeed, the following theorem is well-known, and can be found in, e.g., \cite{roberts1997geometric}.

\begin{theorem} \label{thm:l2}
	For $\mu \in \F'$ and $t \in \mathbb{N}$,
	\begin{equation} \label{ine:l2bound}
	\|\mu K^t - \varpi\|_2 \leq \|\mu - \varpi\|_2 \|K^t\|_2 \leq \|\mu - \varpi\|_2 \|K\|_2^t.
	\end{equation}
\end{theorem}

%\begin{remark}
%	In the above theorem, we use the following notation: for $f \in L^2(\varpi)$ and $c \in \mathbb{R}$, $(f-c)(x) = f(x) - c$ for $x \in \X$.
%\end{remark}

\begin{proof} 
	The second inequality follows from the sub-multiplicity of operator norms.
	We will focus on establishing the first inequality.
	Note that if $f \in \F$ so that $\|f\|_2 = 1$, then $f - \varpi f \in L_0^2(\varpi)$, and $\|f - \varpi f\|_2^2 = \|f\|_2^2 - (\varpi f)^2 \leq 1$.
	(For a constant $c$, $f-c$ means the function satisfying $(f-c)(x) = f(x) - c$.)
	Then
	\[
	\begin{aligned}
		\|\mu K^t - \varpi\|_2 &= \sup_{f \in \F} |(\mu K^t) f - (\varpi K^t) f| \\
		&= \sup_{f \in \F} |(\mu K^t) (f - \varpi f) - (\varpi K^t) (f - \varpi f)| \\
		&= \sup_{f \in \F} \left| \int_{\X^2} [\mu(\df x) - \varpi(\df x)] K^t(x, \df x') [f(x') - \varpi f] \right| \\
		&= \sup_{f \in \F} \left\langle \frac{\df \mu}{\df \varpi} - 1, K^t (f - \varpi f) \right\rangle \\
		&\leq \left\| \frac{\df \mu}{\df \varpi} - 1 \right\|_2 \|K^t\|_2 \sup_{f \in \F}  \|f - \varpi f\|_2 \\
%		&\leq \sup_{f \in \F} \left| \left\langle \frac{\df \mu}{\df \varpi} - 1, f \right\rangle \right| \|K^t\|_2 \quad (\text{for } g \in L_0^2(\varpi), \, \|g\|_2 = \sup_{f \in \F} |\langle f, g \rangle| ) \\
		&\leq \|\mu - \varpi\|_2 \|K^t\|_2.
	\end{aligned}
	\]
	Note that we have used the Cauchy-Schwarz inequality and the definition of the operator norm.
\end{proof}

Theorem~\ref{thm:l2} implies that, if $\|K\|_2 \leq \rho$ for some $\rho < 1$, then asymptotically $\|\mu K^t - \varpi\|_2$ decreases with~$t$ at a geometric rate of $\rho^t$ or faster.
In fact, if the Mtk $K(\cdot, \cdot)$ is reversible with respect to $\varpi(\cdot)$, then the converse is true as well.
Note that $K(\cdot, \cdot)$ is reversible if and only if
\[
\int_{\X^2} f(x) g(x') K(x, \df x') \, \varpi(\df x) = \int_{\X^2} g(x) f(x') K(x, \df x') \, \varpi(\df x)
\]
for $f, g \in L_0^2(\varpi)$, i.e., $\langle f, Kg \rangle = \langle Kf, g \rangle$.
In other words, the Mtk $K(\cdot,\cdot)$ is reversible if and only if the operator $K$ is self-adjoint.
Using this fact we can establish the following result.

\begin{theorem} \citep[][Theorem 2.1]{roberts1997geometric} \label{thm:l2-2}
	Suppose that $K(\cdot,\cdot)$ is reversible.
	Suppose further that there exist $C: \F' \to [0,\infty)$ and $\rho < 1$ such that, for $\mu \in \F'$ and $t \in \mathbb{N}$, 
	\begin{equation} \label{ine:l2convergence}
	\|\mu K^t - \varpi\|_2 \leq C(\mu) \rho^t.
	\end{equation}
	Then $\|K\|_2 \leq \rho$.
\end{theorem}

\begin{proof}
	By assumption, $K$ is a self-adjoint operator.
	Then $K$ has a spectral decomposition: for $f, g \in L_0^2(\varpi)$ and $t \in \mathbb{N}$,
	\[
	\langle K^t f, g \rangle = \int_{-\infty}^{\infty} \lambda^t \langle E_K(\df \lambda) f, g \rangle,
	\]
	where $E_K(\cdot)$ is the spectral measure of $K$, which is a projection-valued measure that is supported on the spectrum of $K$.
	See, e.g., \cite{conway1990course}, \S IX.2; \cite{arveson2006short}, Section~2.7.
	An important property of $E_K(\cdot)$ is that, for a measurable subset $B$ of $\mathbb{R}$, if a non-vanishing function $f \in L_0^2(\varpi)$ is in the range of the orthogonal projection operator $E_K(B)$, then $\|f\|_2^{-2} \langle f, E_K(\cdot) f \rangle$ is a probability measure concentrated on $B$.
	
	Suppose that $\|K\|_2 > \rho$ so that $\|K\|_2 \geq \rho + \varepsilon$ for some $\varepsilon > 0$.
	Then there exists a non-vanishing function $f \in L_0^2(\varpi)$ in the range of the projection operator $E_K([\rho + \varepsilon, \infty) \cup (-\infty, -\rho - \varepsilon])$.
	Let $f_+(x) = \max\{ f(x), 0 \}$ and $f_-(x) = \max\{ -f(x), 0 \}$ for $x \in \X$, so that $f = f_+ - f_-$.
	Let 
	\[
	\mu_+(A) = \frac{2 \int_A f_+(x) \, \varpi(\df x)}{\|f\|_1}, \quad \mu_-(A) = \frac{2 \int_A f_-(x) \, \varpi(\df x)}{\|f\|_1}
	\]
	for $A \in \B$, where
	\[
	\|f\|_1 = \int_{\X} |f(x)| \, \varpi(\df x) = 2 \int_{\X} f_+(x) \, \varpi(\df x) = 2 \int_{\X} f_-(x) \, \varpi(\df x).
	\]
	Then $\mu_+$ and $\mu_-$ are in $\F'$.
	Moreover, $f/\|f\|_2$ is in $\F$.
	By the spectral decomposition, for $t \in \mathbb{N}$,
	\[
	\begin{aligned}
		\|\mu_+ K^{2t} - \mu_- K^{2t} \|_2 &\geq \int_{\X^2} [\mu_+(\df x) - \mu_-(\df x)] K^{2t}(x, \df x') \frac{f(x')}{\|f\|_2} \\
		&= 2 \left\langle \frac{f_+ - f_-}{\|f\|_1}, K^{2t} \frac{f}{\|f\|_2} \right\rangle \\
		&= \frac{2}{\|f\|_1 \|f\|_2} \int_{-\infty}^{\infty} \lambda^{2t} \langle f, E_K(\df \lambda) f \rangle \\
		&\geq \frac{2\|f\|_2}{\|f\|_1} (\rho + \varepsilon)^{2t} .
	\end{aligned}
	\]
	The last line holds because the probability measure $\|f\|_2^{-2} \langle f, E_K(\cdot) f \rangle$ is concentrated on $[\rho + \varepsilon, \infty) \cup (-\infty, -\rho - \varepsilon]$.
	The $L^2$ distance satisfies the triangle inequality, so
	\[
	\|\mu_+ K^t - \varpi\|_2 + \|\mu_- K^t - \varpi\|_2 \geq \|\mu_+ K^t - \mu_- K^t \|_2 \geq \frac{2\|f\|_2}{\|f\|_1} (\rho + \varepsilon)^t
	\]
	for $t \in \mathbb{N}$.
	Then it is impossible for~\eqref{ine:l2convergence} to hold for all $\mu \in \F'$ and $t \in \mathbb{N}$.
	Thus, it must hold that $\|K\|_2 \leq \rho$.
\end{proof}

Theorems~\ref{thm:l2} and~\ref{thm:l2-2} show that, if $K(\cdot,\cdot)$ is reversible, then $\|K\|_2$, which lies in $[0,1]$, can be regarded the convergence rate of the corresponding chain.
The smaller $\|K\|_2$ is, the faster the chain converges.

In some simple scenarios, it is possible to calculate $\|K\|_2$ directly using functional analytic techniques.
For instance, in Example \ref{ex:gaussian}, it can be shown, using orthogonal polynomials, that $\|K_{p,\alpha}\|_2 = \alpha$ \citep[see, e.g.,][]{diaconis2008gibbs}.
One can compare this rate with that in Section~\ref{sec:coupling} involving the total variation distance.
In more complex scenarios, one hopes to derive some reasonably sharp bounds on $\|K\|_2$.
Of course, to apply Theorem \ref{thm:l2} in practice, one would also need to get a handle on $\|\mu - \varpi\|_2$.
But since $\|K\|_2$ is the more important quantity in \eqref{ine:l2bound} when $t$ is large, it will be our focus herein.

We end Section \ref{ssec:l2basic} with an elementary method for bounding $\|K\|_2$ from above which can be applied to Example \ref{ex:IMH}.
More sophisticated methods will be given later.

\begin{theorem} \label{thm:l2-elementary}
	Suppose that there exists $\varepsilon > 0$ such that $K(x,A) \geq \varepsilon \varpi(A)$ for $x \in \X$ and $A \in \B$.
	Then $\|K\|_2 \leq 1 - \varepsilon$.
\end{theorem}
\begin{proof}
	It is clear that $\varepsilon \leq 1$.
	If $\varepsilon = 1$ then $Kf = \varpi f = 0$ for $f \in L_0^2(\varpi)$, indicating that $\|K\|_2 = 0$.
	Suppose that $\varepsilon < 1$.
	Let $R(x,A) = (1-\varepsilon)^{-1} [K(x,A) - \varepsilon \varpi(A)]$ for $x \in \X$ and $A \in \B$.
	Then $R(\cdot,\cdot)$ is an Mtk such that $\varpi R = \varpi$, and we may view $R$ as a Markov operator on $L_0^2(\varpi)$.
	Replacing $K(\cdot,\cdot)$ with $R(\cdot,\cdot)$ in \eqref{ine:Jensen}, we find that $\|R\|_2 \leq 1$.
	Thus, for $f \in L_0^2(\varpi)$, 
	\[
	\|Kf\|_2 = \|\varepsilon \varpi f + (1-\varepsilon) R f\|_2 = (1-\varepsilon) \|Rf\|_2 \leq (1-\varepsilon) \|f\|_2.
	\]
	This implies that $\|K\|_2 \leq 1 - \varepsilon$.
\end{proof}

Let us apply Theorem \ref{thm:l2-elementary} to the independent Metropolis Hastings chain in Example \ref{ex:IMH}.
Recall from \eqref{ine:doeblin-IMH} that, for $x \in \X= [0,1]$ and $A \in \B$, $K_s(x,A) \geq (1/M_s) \pi_s(A)$.
Hence, by Theorem \ref{thm:l2-elementary}, $\|K_s\|_2 \leq 1 - 1/M_s$.
In Section \ref{ssec:lower}, it will be shown that this bound is tight.

\subsection{The spectral gap and the conductance} \label{ssec:conductance}

Let $K(\cdot,\cdot)$ be reversible, so that the corresponding operator is self-adjoint.
Consider the task of bounding $\|K\|_2$, particularly from above.
Self-adjoint-ness implies that
\begin{equation} \label{eq:Knorm}
\|K\|_2 = \max \left\{ \sup_{f \in L_0^2(\varpi), \, f \neq 0} \frac{\langle f, Kf \rangle}{\|f\|_2^2}, - \inf_{f \in L_0^2(\varpi), \, f \neq 0} \frac{\langle f, Kf \rangle}{\|f\|_2^2} \right\}
\end{equation}
\citep[see, e.g.,][\S 14, Corollary 5.1]{helmberg2014introduction}.
In certain cases, it is possible to bound the second term in the maximum quite easily.
For instance, if~$K$ is associated with a random-scan Gibbs algorithm or a data augmentation algorithm, then~$K$ is positive semi-definite in the sense that it is self-adjoint, and $\langle f, Kf \rangle \geq 0$ for $f \in L_0^2(\varpi)$, so the second term is at most zero.
See Theorem 3 of \cite{liu1995covariance} and Theorem 3.2 of \cite{liu:wong:kong:1994}.
This is also the case if $K(\cdot,\cdot)$ is the two-step Mtk of some reversible chain, i.e., $K = T^2$ for some Mtk $T(\cdot,\cdot)$ that is reversible with respect to $\varpi(\cdot)$.
Finally, if the chain is lazy in the sense that $K(x, \{x\}) \geq c$ for every $x$ and some positive constant $c$, then the second term is at most $1-2c$.
For various MCMC algorithms, much effort has been spent on bounding the first term from above, or equivalently, bounding the spectral gap, defined as
\[
G(K) = 1 - \sup_{f \in L_0^2(\varpi), \, f \neq 0} \frac{\langle f, Kf \rangle}{\|f\|_2^2},
\]
from below.

It is worth mentioning that the spectral gap is also closely related to the asymptotic variance of Monte Carlo estimators.
To be more precise, let $(X_t)_{t=0}^{\infty}$ be a chain associated with~$K$, and let $f \in L^2(\varpi)$.
Then, under regularity conditions, $n^{1/2}[ n^{-1} \sum_{i=1}^n f(X_i) - \varpi f]$ is asymptotically normally distributed as $n \to \infty$, and the asymptotic variance is upper bounded by $[2-G(K)]/G(K)$.
See, e.g., \cite{chan1994discussion}, equation (7).

One important approach for bounding the spectral gap is relating it to a quantity called the ``conductance" \citep{jerrum1988conductance}.
The conductance of $K$ is defined to be
\[
\Phi_K = \inf_{A \in \B, \, 0 < \varpi(A) < 1} \phi_K(A),  \text{ where } \phi_K(A) = \frac{\int_A \varpi(\df x) K(x,A^c) }{\varpi(A) \varpi(A^c)}.
\]
Loosely speaking, $\phi_K(A)$ measures the probability flow from~$A$ to its complement $A^c$, after adjusting for the probability masses of $A$ and $A^c$.
A large conductance indicates that the chain can freely move around the state space~$\X$, and vice versa.

The conductance is related to the spectral gap through the following remarkable result, called Cheeger's inequality.

\begin{theorem} \citep[][Theorem 2.1]{lawler1988bounds} \label{thm:cheeger}
	For the reversible Mtk $K(\cdot,\cdot)$, $\Phi_K^2/8 \leq G(K) \leq \Phi_K$.
\end{theorem}

\begin{proof}
	For $A \in \B$ and $x \in \X$, let $\eta_A(x) = \ind_{x \in A} - \varpi(A)$.
	Then $\eta_A \in L_0^2(\varpi)$, and it is straightforward to check that
	\[
	\phi_K(A) = 1 - \frac{\langle \eta_A, K \eta_A \rangle}{\|\eta_A\|_2^2}, \quad A \in \B.
	\]
	Then $G(K) \leq \phi_K(A)$ for any $A$ and thus, $G(K) \leq \Phi_K$.
	
	We now establish the other inequality.
	Using the fact that $\varpi K = \varpi$, one can obtain
	\begin{equation} \label{eq:cheeger-0}
	\begin{aligned}
		G(K) &= \inf_{f \in L_0^2(\varpi), \, f \neq 0} \frac{ \int_{\X} f(x)^2 \, \varpi(\df x) - \int_{\X^2} f(x) f(x') K(x, \df x') \, \varpi(\df x) }{\|f\|_2^2} \\
		&= \frac{1}{2} \inf_{f \in L_0^2(\varpi), \, f \neq 0} \frac{ \int_{\X^2} [f(x) - f(x')]^2 K(x, \df x') \, \varpi(\df x) }{\|f\|_2^2}
	\end{aligned}
	\end{equation}
	For now, fix $f \in L_0^2(\varpi)$ such that $f \neq 0$ and $s \in \mathbb{R}$. 
	Let $f_s(x) = f(x) - s$ for $x \in \X$.
%	The inner product of~$f$ and any constant function is zero, so
%	\begin{equation} \label{eq:cheeger-1}
%	\|f_s\|_2^2 = \|f - s\|_2^2 = \|f\|_2^2 + s^2 \geq \|f\|_2^2.
%	\end{equation}
	By the Cauchy-Schwarz inequality, 
	\begin{equation} \label{eq:cheeger-2}
	\begin{aligned}
		\int_{\X^2} [f(x) - f(x')]^2 K(x, \df x') \, \varpi(\df x) &= \int_{\X^2} [f_s(x) - f_s(x')]^2 K(x, \df x') \, \varpi(\df x) \\
		&\geq \frac{\left\{ \int_{\X^2} |f_s(x)^2 - f_s(x')^2| K(x, \df x') \, \varpi(\df x) \right\}^2 }{ \int_{\X^2} [f_s(x) + f_s(x')]^2 K(x, \df x') \, \varpi(\df x) } \\
		&\geq \frac{\left\{ \int_{\X^2} |f_s(x)^2 - f_s(x')^2| K(x, \df x') \, \varpi(\df x) \right\}^2 }{ 2 \int_{\X^2} [f_s(x)^2 + f_s(x')^2] K(x, \df x') \, \varpi(\df x) } \\
		&= \frac{\left\{ \int_{\X^2} |f_s(x)^2 - f_s(x')^2| K(x, \df x') \, \varpi(\df x) \right\}^2 }{ 4 \|f_s\|_2^2 }
	\end{aligned}
	\end{equation}
	Let $A_{s,f}(t) = \{ x: \, f_s(x)^2 \geq t \}$ for $t \in [0,\infty)$.
	Then
	\begin{equation} \label{eq:cheeger-3}
	\begin{aligned}
		&\int_{\X^2} |f_s(x)^2 - f_s(x')^2| K(x, \df x') \, \varpi(\df x) \\
		=& \int_{\X^2} \ind_{f_s(x')^2 < f_s(x)^2} \int_{f_s(x')^2}^{f_s(x)^2}  \df t \, K(x, \df x') \, \varpi(\df x) + \int_{\X^2} \ind_{f_s(x)^2 < f_s(x')^2} \int_{f_s(x)^2}^{f_s(x')^2}  \df t \, K(x, \df x') \, \varpi(\df x) \\
		=& \int_0^{\infty} \int_{\X^2} \ind_{f_s(x')^2 < t \leq f_s(x)^2}  \, K(x, \df x') \, \varpi(\df x) \, \df t + \int_0^{\infty} \int_{\X^2} \ind_{f_s(x)^2 < t \leq f_s(x')^2}   \, K(x, \df x') \, \varpi(\df x) \, \df t \\
		=& \int_0^{\infty} \int_{A_{s,f}(t)} \varpi(\df x) K(x, A_{s,f}(t)^c) \, \df t + \int_0^{\infty} \int_{A_{s,f}(t)^c} \varpi(\df x) K(x, A_{s,f}(t)) \, \df t \\
		\geq& 2 \Phi_K \int_0^{\infty} \varpi(A_{s,f}(t)) [1 - \varpi(A_{s,f}(t))] \, \df t.
	\end{aligned}
	\end{equation}
	A similar calculation reveals that
	\begin{equation} \label{eq:cheeger-4}
	\int_{\X^2} |f_s(x)^2 - f_s(x')^2| \, \varpi(\df x') \, \varpi(\df x) = 2  \int_0^{\infty} \varpi(A_{s,f}(t)) [1 - \varpi(A_{s,f}(t))] \, \df t.
	\end{equation}
	Letting~$f$ and~$s$ vary, we have the following:
	\begin{equation} \nonumber
	\begin{aligned}
		G(K) &\geq \inf_{f \in L_0^2(\varpi), \, f \neq 0} \sup_{s \in \mathbb{R}} \frac{ \left\{ \int_{\X^2} |f_s(x)^2 - f_s(x')^2| K(x, \df x') \, \varpi(\df x) \right\}^2 }{8 \|f\|_2^2 \|f_s\|_2^2 } \quad \text{by } \eqref{eq:cheeger-0} \text{ and } \eqref{eq:cheeger-2} \\
		& \geq \frac{\Phi_K^2}{8} \inf_{f \in L_0^2(\varpi), \, f \neq 0} \sup_{s \in \mathbb{R}} \left( \frac{\int_{\X^2} |f_s(x)^2 - f_s(x')^2| \, \varpi(\df x') \, \varpi(\df x) }{\|f\|_2 \|f_s\|_2 } \right)^2 \quad \text{by } \eqref{eq:cheeger-3} \text{ and } \eqref{eq:cheeger-4} \\
		&= \frac{\Phi_K^2}{8} \inf_{f \in L_0^2(\varpi), \, f \neq 0} \sup_{s \in \mathbb{R}} \left( \frac{E[|(X-s)^2 - (Y-s)^2|]}{ \sqrt{\mbox{var} (X)} \sqrt{E[(X - s)^2]} } \right)^2,
	\end{aligned}
	\end{equation}
	where $X$ and $Y$ are independently and identically (iid) distributed as $\varpi \circ f^{-1}$, i.e., the distribution of $f(W)$ with $W \sim \varpi$.
	Note that $E(X) = \varpi f = 0$, and $\mbox{var}(X) = \|f\|_2^2 \in (0,\infty)$.
	
	To conclude the proof, it suffices to show that, for two iid random variables $X$ and $Y$ with mean zero and some standard deviation $\sigma > 0$, 
	\begin{equation} \label{ine:cheeger-a}
	\sup_{s \in \mathbb{R}}  \frac{E[|(X-s)^2 - (Y-s)^2|]}{\sigma \sqrt{E[(X - s)^2]} }  \geq 1.
	\end{equation}
	By the dominated convergence theorem, $\lim_{s \to \infty} E[s^{-2}(X-s)^2] = 1$, while
	\[
	\begin{aligned}
		\lim_{s \to \infty} E[s^{-1}|(X-s)^2 - (Y-s)^2|] &= 2 E[|X-Y|] \\
		&= 2 E[E(|X-Y| \mid X)] \\
		&\geq 2 E[| E(X-Y \mid X) |] \\
		&= 2 E(|X|) \quad \text{since } E(Y \mid X) = E(Y) = 0.
	\end{aligned}
	\]
	Thus,
	\begin{equation} \label{ine:cheeger-b}
	\sup_{s \in \mathbb{R}}  \frac{E[|(X-s)^2 - (Y-s)^2|]}{\sqrt{E[(X - s)^2]}}  \geq \lim_{s \to \infty} \frac{E[|(X-s)^2 - (Y-s)^2|]}{\sqrt{E[(X - s)^2]}} \geq 2 E(|X|).
	\end{equation}
	On the other hand, using the assumption that $E(X^2) = E(Y^2) = \sigma^2$ and the fact that $|u^2 - \sigma^2| \geq (|u|-\sigma)^2$ for $u \in \mathbb{R}$, we have
	\[
	E(|X^2 - Y^2|) \geq E[|E(X^2 - Y^2 \mid X)|] = E(|X^2 - \sigma^2|) \geq E[(|X|-\sigma)^2] = 2 \sigma^2 - 2 \sigma E(|X|).
	\]
	Thus,
	\begin{equation} \label{ine:cheeger-c}
	\sup_{s \in \mathbb{R}}  \frac{E[|(X-s)^2 - (Y-s)^2|]}{\sqrt{E[(X - s)^2]}}  \geq \frac{E|(X-0)^2 - (Y-0)^2|}{\sqrt{E[(X-0)^2]}} \geq 2\sigma - 2E(|X|).
	\end{equation}
	Combining \eqref{ine:cheeger-b} and~\eqref{ine:cheeger-c} gives \eqref{ine:cheeger-a}.
\end{proof}

\subsection{Bounds via isoperimetric inequalities} \label{ssec:isoperimatric}

We now describe a method for bounding $\Phi_K$, and in turn, $G(K)$ from below.
The method is particularly powerful when $\varpi(\cdot)$ admits a log-concave density function.
It is based on a certain type of isoperimetric inequality.
Let $\mbox{dist}: \B^2 \to [0,\infty]$ be some function that quantifies how far two sets are from each other.
We say $\varpi(\cdot)$ satisfies a three-set isoperimetric inequality of Cheeger type if one can find some $\delta \in (0,\infty)$ and $\kappa \in (0,\infty)$ such that, for any partition of $\X$ consisting of three measurable sets, say, $\{S_1, S_2, S_3\}$,
\begin{equation} \label{ine:iso}
 \varpi(S_3) \geq 
 	\kappa \, \varpi(S_1) \varpi(S_2)  \text{ whenever } \mbox{dist}(S_1, S_2) \geq \delta.
\end{equation}
Ideally, $\kappa$ is not close to zero, especially if $\delta$ is large.
This would indicate that, if two sets $S_1$ and $S_2$ are not too close, then $S_3 = (S_1 \cup S_2)^c$ must have a non-negligible probability mass relative to the masses of $S_1$ and $S_2$.
Loosely speaking, this means that the state space~$\X$ cannot exhibit two disjoint subdomains, each possessing substantial probability mass, where transitioning from one subdomain to the other necessitates covering extensive distances through a low-probability region.

\begin{figure}
\begin{center}
 {\includegraphics[height=2in]{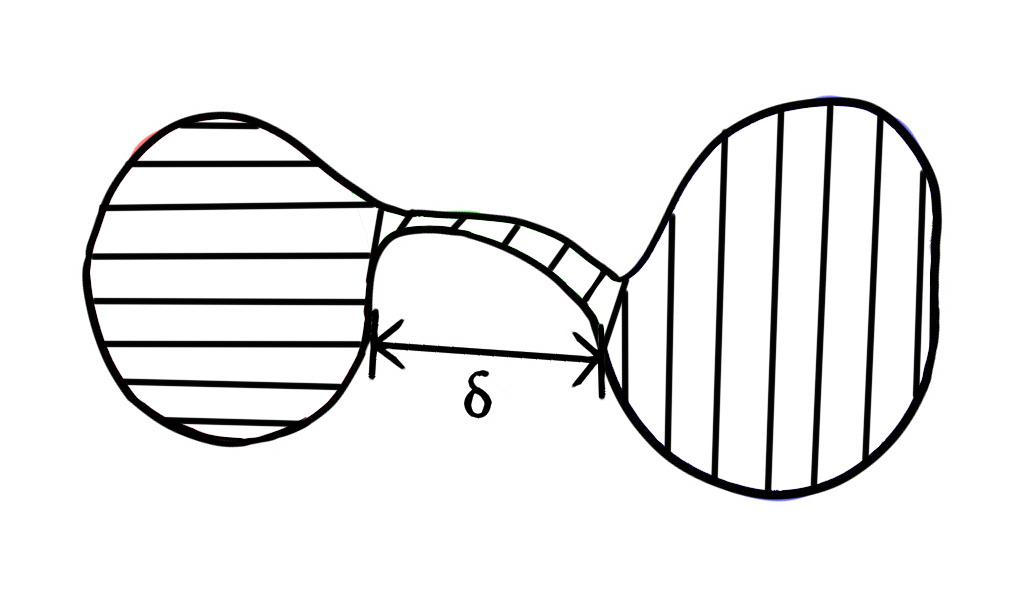}} 
 \caption{A dumbbell-shaped domain partitioned into three parts.
 The sets $S_1$ (horizontal stripes) and $S_2$ (vertical stripes) are separated by a narrow corridor $S_3$ (diagonal stripes).
 } \label{fig:isoperimetric}
\end{center}
\end{figure}  
Figure~\ref{fig:isoperimetric} shows a scenario where a good isoperimetric inequality, i.e., one with a large value of $\kappa$, would eliminate.
Here, a dumbbell-shaped domain~$\X$ is presented.
Let $\varpi(\cdot)$ be the uniform distribution on~$\X$.
We can partition~$\X$ into three regions, $S_1$, $S_2$, and $S_3$, which are filled with, respectively, horizontal, vertical, and diagonal stripes.
Imagine that $\delta = \mbox{dist}(S_1, S_2)$, and $\varpi(S_3)$ is small relative to $\varpi(S_1) \varpi(S_2)$.
Then \eqref{ine:iso} cannot hold for large values of $\kappa$.
Scenarios like this could prevent the fast mixing of some Markov chains whose stationary distributions are $\varpi(\cdot)$.
Indeed, it may be difficult for a chain to travel from $S_1$ to $S_2$ due to how narrow the corridor connecting the two sets is.

There is a large literature devoted to establishing isoperimetric inequalities, especially for distributions on Euclidean spaces with log-concave density functions.
We will not attempt to establish an inequality ourselves since this typically requires some sophisticated analysis.
Instead, we state, without proof, a well-known inequality for distributions with strongly log-concave density functions \citep[see, e.g.,][]{ledoux1999concentration,bobkov2003localization}.

\begin{theorem} \citep[][Theorem~1]{bobkov2003localization} \label{thm:iso}
	Let $\X = \mathbb{R}^p$ for some $p \in \mathbb{N}_+$.
	Suppose that $\varpi$ admits a probability density function with respect to the Lebesgue measure that is proportional to $\exp[-\|x\|^2/(2\sigma^2)] g(x)$, where $\sigma > 0$, and $x \mapsto g(x)$ is log-concave, i.e., for all $x, y \in \mathbb{R}^p$ and $\lambda \in (0,1)$, the inequality $g(\lambda x + (1-\lambda)y) \geq g(x)^{\lambda} g(y)^{1-\lambda}$ holds.
	Let $\{S_1, S_2, S_3\}$ be a measurable partition of $\mathbb{R}^p$.
	Then, for $i = 1,2$,
	\[
	\varpi(S_3) \geq F_N \left(F_N^{-1} (\varpi(S_i)) + \frac{\mbox{dist} \,(S_1, S_2)}{\sigma} \right) - \varpi(S_i),
	\]
	where $F_N(\cdot)$ is the cumulative distribution function of the standard normal distribution, and $\mbox{dist}(S_1, S_2) = \inf_{x \in S_1, y \in S_2} \|x-y\|$.
\end{theorem}

We can get an inequality of the form \eqref{ine:iso} from Theorem~\ref{thm:iso} using some calculus.
Let $r = \min \{\varpi(S_1), \varpi(S_2)\}$ so that $F_N^{-1}(r) \leq 0$, and denote by $f_N(\cdot)$ the density function of the standard normal distribution.
Under the assumption of Theorem \ref{thm:iso}, if $\mbox{dist}(S_1, S_2) \geq \delta$ for some $\delta \in (0,\infty)$, then
\[
\begin{aligned}
	\varpi(S_3) &\geq \int_{F_N^{-1}(r)}^{F_N^{-1}(r) + \delta/\sigma} f_N(t) \, \df t \\
	&\geq \begin{cases}
		f_N \left(F_N^{-1}(r) + \delta/\sigma \right) \, \delta/\sigma & \text{if }
		F_N^{-1}(r) + \delta/\sigma \geq - F_N^{-1}(r) \geq 0 \\
		f_N(F_N^{-1}(r)) \, \delta/\sigma & \text{otherwise}
	\end{cases} \\
	&\geq \begin{cases}
		f_N \left(\delta/\sigma \right) \, \delta/\sigma & \text{if }
		F_N^{-1}(r) + \delta/\sigma \geq - F_N^{-1}(r) \geq 0 \\
		f_N(F_N^{-1}(r)) \, \delta/\sigma & \text{otherwise}
	\end{cases} \\
	&\geq \sqrt{2\pi} \, f_N(F_N^{-1}(r)) \, f_N \left( \delta/\sigma \right) \frac{\delta}{\sigma}.
\end{aligned}
\]
By (4) in \cite{sampford1953some}, one can show that $f_N(q)/[1-F_N(q)] \geq 4 F_N(q)/\sqrt{2\pi}$ for $q \geq 0$.
Letting $q = -F_N^{-1}(r)$ yields $f_N(F_N^{-1}(r)) \geq 4 r (1-r)/\sqrt{2\pi}$.
Moreover, $r(1-r) \geq \varpi(S_1) \varpi(S_2)$.
Thus, for an arbitrary choice of $\delta \in (0,\infty)$,
\[
	\varpi(S_3) \geq \frac{4 \delta f_N \left( \delta/\sigma \right)  }{\sigma}  \varpi(S_1) \varpi(S_2)  \text{ whenever } \mbox{dist}(S_1, S_2) \geq \delta,
\]
i.e., \eqref{ine:iso} holds with $\delta \in (0,\infty)$ and $\kappa = 4 (\delta/\sigma) f_N(\delta/\sigma) $.

For other examples of three-set isoperimetric inequalities, see, e.g., Theorem 2.6 of \cite{lovasz1993random}, Theorem 2.1 of \cite{kannan1996sampling}, Theorem 1 of \cite{lovasz1999hit}, and Theorem 4.2 of \cite{cousins2014cubic}.
Three-set isoperimetric inequalities can also be obtained through more standard forms of isoperimetric inequalities that involve the perimeter and volume of an arbitrary measurable set \citep[see, e.g.,][]{bobkov1995some,andrieu2022explicit}.

We now give a bound on $G(K)$ based on a three-set isoperimetric inequality of Cheeger type.
It is largely similar to existing results from, e.g., \cite{lovasz1999hit}, \cite{belloni2009computational}, and \cite{dwivedi2019log}.  

\begin{theorem} \label{thm:isogap}
	Let $\dist': \X^2 \to [0,\infty)$ be a measurable function (not necessarily a metric), and for $A, B \in \B$, let
	\[
	\mbox{dist}(A,B) = \inf_{x \in A, \, y \in B} \dist'(x,y).
	\]
	Suppose that there exist $\delta \in (0,\infty)$ and $\varepsilon \in (0,1]$ such that the following ``close coupling condition" holds:
	\begin{equation} \label{ine:closecoupling}
		\|\delta_x K - \delta_y K \|_{\tv} \leq 1 - \varepsilon \; \text{ whenever } \dist'(x,y) < \delta.
	\end{equation}
	(Recall that $\delta_x$ is the point mass at $x$, so $\delta_x K(\cdot) = K(x,\cdot)$.)
	Suppose further that $\varpi(\cdot)$ satisfies a three-set isoperimetric inequality of Cheeger type with $\delta$ given above and some $\kappa \in (0,\infty)$. 
	Then, for $A \in \B$ such that $\varpi(A) \in (0,1)$ and $a \in (0,1)$,
%	\begin{equation} \label{ine:isogap-0}
%	\phi_K(A) := \frac{\int_A \varpi(\df x) K(x, A^c)}{\varpi(A) \varpi(A^c)} \geq  \frac{[\ell(\delta) + 1 - \sqrt{1+2\ell(\delta)}] \varepsilon}{2 \ell(\delta)}.
%	\end{equation}
	\begin{equation} \label{ine:isogap-0}
		\phi_K(A) := \frac{\int_A \varpi(\df x) K(x, A^c)}{\varpi(A) \varpi(A^c)} \geq  \varepsilon \min \left\{ \frac{1-a}{2}, \frac{a^2 \kappa}{4} \right\}.
	\end{equation}
	This inequality holds even when $K(\cdot,\cdot)$ is non-reversible.
\end{theorem}
\begin{proof}
	Let $A \in \B$ be such that $\varpi(A) \in (0,1)$.
	Let 
	\[
	S_1 = \{ x \in A: \, K(x, A^c) < \varepsilon/2 \}, \quad S_2 = \{ x \in A^c: \, K(x, A) < \varepsilon/2 \}, \quad S_3 = (S_1 \cup S_2)^c.
	\]
	See Figure~\ref{fig:isoperimetric-proof}.
	Fix $a \in (0,1)$.
	We will establish \eqref{ine:isogap-0} in three cases: (i) $\varpi(S_1) \leq a \varpi(A)$, (ii) $\varpi(S_2) \leq a \varpi(A^c)$, and (iii) $\varpi(S_1) > a \varpi(A)$ and $\varpi(S_2) > a \varpi(A^c)$.
	Note that these three cases exhaust all possibilities.
	
	\begin{figure}
		\begin{center}
			{\includegraphics[height=2in]{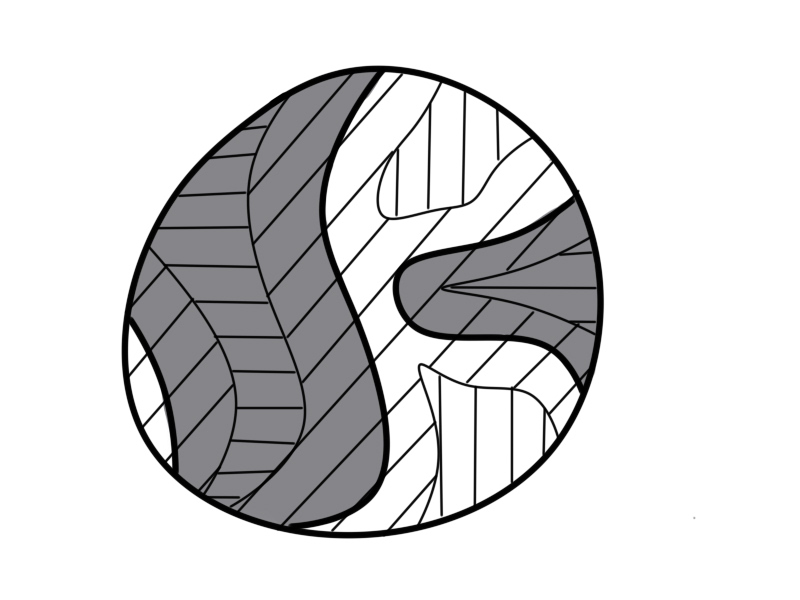}} 
			\caption{A domain partitioned into $A$ (grey) and $A^c$.
				Depending on the amount of probability flow from each point in $A$ (resp. $A^c$) to $A^c$ (resp. $A$), the domain can be alternatively partitioned into
				$S_1$ (horizontal stripes), $S_2$ (vertical stripes), and $S_3$ (diagonal stripes).
			} \label{fig:isoperimetric-proof}
		\end{center}
	\end{figure}
	
	{\noindent \bf Case (i):}
	By the definition of $S_3$,
	\begin{equation} \label{ine:isogap-1}
		\begin{aligned}
			\int_A \varpi(\df x) K(x, A^c)  \geq \int_{S_3 \cap A} \varpi(\df x) K(x, A^c) 
			\geq \frac{\varepsilon}{2} \varpi(S_3 \cap A) .
		\end{aligned}
	\end{equation}
	In Case (i),
	\[
	\varpi(S_3 \cap A) = \varpi(A) - \varpi(S_1) \geq (1-a) \varpi(A) \geq (1-a) \varpi(A) \varpi(A^c) ,
	\]
	so, by \eqref{ine:isogap-1}, we have \eqref{ine:isogap-0}.
	
	{\noindent \bf Case (ii):}
	Since $\varpi K = \varpi$,
	\[
	\begin{aligned}
		\int_A \varpi(\df x) K(x, A^c) &= \int_{\X} \varpi(\df x) K(x,A^c) - \int_{A^c} \varpi(\df x) K(x, A^c) \\
		&= \varpi(A^c) - \int_{A^c} \varpi(\df x) [1 - K(x,A)] \\
		&= \int_{A^c} \varpi(\df x) K(x,A).
	\end{aligned}
	\]
	Then
	\begin{equation} \label{ine:isogap-2}
		\int_A \varpi(\df x) K(x, A^c) \geq \int_{S_3 \cap A^c} \varpi(\df x) K(x,A) \geq \frac{\varepsilon}{2} \varpi(S_3 \cap A^c) .
	\end{equation}
	In Case (ii),
	\[
	\varpi(S_3 \cap A^c) = \varpi(A^c) - \varpi(S_2) \geq (1-a) \varpi(A^c) \geq (1-a) \varpi(A) \varpi(A^c),
	\]
	so, by \eqref{ine:isogap-2}, we have \eqref{ine:isogap-0}.

	{\noindent \bf Case (iii):}	
%	In this case, $\varpi(S_1)$ and $\varpi(S_2)$ are nonzero, so $S_1$ and $S_2$ are nonempty.
	By the definition of the total variation distance (see Section \ref{sec:intro}), for $x \in S_1$ and $y \in S_2$,
	\[
	\|\delta_x K - \delta_y K\|_{\tv} \geq K(x, A) - K(y,A) > 1 - \varepsilon.
	\]
	By the close coupling condition \eqref{ine:closecoupling}, $\dist'(x,y) \geq \delta$ for $x \in S_1$ and $y \in S_2$, so $\mbox{dist}(S_1, S_2) \geq \delta$.
	By the isoperimetric inequality,
	\begin{equation} \nonumber
	\varpi(S_3) \geq \kappa \varpi(S_1) \varpi(S_2).
	\end{equation}
	Note that \eqref{ine:isogap-1} and \eqref{ine:isogap-2} still hold.
	Combining them with the display above yields
	\begin{equation} \label{ine:isogap-3}
	\int_A \varpi(\df x) K(x, A^c) \geq \frac{\varepsilon \varpi(S_3)}{4} \geq \frac{\varepsilon \kappa }{4} \varpi(S_1) \varpi(S_2)
	\end{equation}
	In Case (iii),
	\[
	\varpi(S_1) \varpi(S_2) \geq a^2 \varpi(A) \varpi(A^c),
	\]
	so, by \eqref{ine:isogap-3}, we have \eqref{ine:isogap-0}.

\end{proof}

\subsubsection{Application to the Gaussian chain}

Recall that, in Example \ref{ex:gaussian}, $\X = \mathbb{R}^p$, $\varpi_p(\cdot)$ is the $\mbox{N}_p(0, I_p)$ distribution, and $K_{p,\alpha}(x,\cdot)$ is the $\mbox{N}_p(\alpha x, (1-\alpha^2) I_p)$ distribution for $x \in \mathbb{R}^p$, where $\alpha \in [0,1)$.
We can place an upper bound on $\|K_{p,\alpha}\|_2$ using Theorems \ref{thm:cheeger}, \ref{thm:iso}, and \ref{thm:isogap}.

In light of the discussion in Section \ref{ssec:conductance}, we first show that $K_{p,\alpha}$, as a linear operator on $L_0^2(\varpi_p)$, is positive semi-definite.
Recall first that the Mtk $K_{p,\alpha}(\cdot,\cdot)$ is reversible with respect to $\varpi_p(\cdot)$.
This is equivalent to $K_{p,\alpha}$ being self-adjoint.
Next, note that, for $f \in L_0^2(\varpi_p)$ and $x \in \mathbb{R}^p$,
\[
\begin{aligned}
	K_{p,\alpha} f(x) &= \frac{1}{[2\pi (1- \alpha^2)]^{p/2}} \int_{\mathbb{R}^p} f(x')  \exp \left[ - \frac{1}{2(1-\alpha^2)} \| x' - \alpha x\|^2 \right] \df x' \\
	&= \frac{1}{[2\pi (1- \alpha)]^{p}} \int_{\mathbb{R}^p} f(x') \int_{\mathbb{R}^p}  \exp \left[  - \frac{\|x'-\sqrt{\alpha} x''\|^2}{2(1-\alpha)} - \frac{\|x''-\sqrt{\alpha} x\|^2}{2(1-\alpha)} \right] \df x'' \df x' \\
	&= K_{p,\sqrt{\alpha}}^2 f(x).
\end{aligned}
\]
Since $K_{p,\sqrt{\alpha}}(\cdot, \cdot)$ is also reversible with respect to $\varpi_p(\cdot)$, the corresponding operator is self-adjoint.
As a result, for $f \in L_0^2(\varpi_p)$,
\[
\langle f, K_{p,\alpha} f \rangle = \langle f, K_{p,\sqrt{\alpha}} K_{p,\sqrt{\alpha}} f \rangle = \langle K_{p,\sqrt{\alpha}} f, K_{p,\sqrt{\alpha}} f \rangle = \|K_{p,\sqrt{\alpha}} f\|_2^2 \geq 0.
\]
Hence, $K_{p,\alpha}$ is positive semi-definite.
By \eqref{eq:Knorm} and Cheeger's inequality (Theorem \ref{thm:cheeger}),
\[
\|K_{p,\alpha}\|_2 = 1 - G(K_{p,\alpha}) \leq 1 - \Phi_{K_{p,\alpha}}^2/8.
\]

We now bound $\Phi_{K_{p,\alpha}}$ from below.
We may pretend that the only thing we know about $\varpi_p(\cdot)$ is the following: It has a density function of the form $e^{-h_p(x)}$, where $\nabla^2 h_p(x) - I_p$ is positive semi-definite for $x \in \mathbb{R}^p$.
($\nabla^2 h_p(x)$ denotes the Hessian matrix of $h_p$.)
Then, by Theorem \ref{thm:iso}, $\varpi_p(\cdot)$ satisfies a three-set isoperimetric inequality of Cheeger type with an arbitrary positive $\delta$ and $\kappa = 4 \delta f_N(\delta)$.
On the other hand, by \eqref{eq:tv} and \eqref{ine:normaltv}, for $x, y \in \mathbb{R}^p$,
\[
\begin{aligned}
	\| \delta_x K - \delta_y K \|_{\tv} = 1 - \int_{\mathbb{R}^p} \min\{ k_{p,\alpha}(x, \df x') , k_{p,\alpha}(y, \df x') \} \, \df x' = 1 - 2 F_N \left( - \frac{\alpha \|x-y\|}{ 2 \sqrt{1-\alpha^2} } \right),
\end{aligned}
\]
where $k_{p,\alpha}(x,\cdot)$ is the density function of $K_{p,\alpha}(x,\cdot)$.
Applying the bound in Theorem \ref{thm:isogap} with $a = 1/2$ and $\delta = \sqrt{1-\alpha^2}/\alpha$ yields
\[
\begin{aligned}
	\Phi_{K_{p,\alpha}} &\geq \min \left\{ \frac{1}{4}, \frac{ \delta f_N(\delta)  }{4} \right\} \times 2 F_N \left( - \frac{1}{2} \right)  = \frac{\sqrt{1-\alpha^2}}{2\sqrt{2\pi} \alpha} \exp \left( - \frac{1-\alpha^2}{2\alpha^2} \right) F_N \left( - \frac{1}{2} \right).
\end{aligned}
\]
Thus,
\[
\|K_{p,\alpha}\|_2 \leq 1 - \frac{1 - \alpha^2}{64 \pi \alpha^2} \exp \left( - \frac{1-\alpha^2}{\alpha^2} \right) F_N \left( - \frac{1}{2} \right)^2.
\]
Recall that, in truth, $\|K_{p,\alpha}\|_2 = \alpha$.
The bound correctly indicates that $\|K_{p,\alpha}\|_2$ is bounded away from unity as $p \to \infty$, and that $(1 - \|K_{p,\alpha}\|_2)/(1 - \alpha)$ is bounded away from zero as $\alpha \to 1$.

\subsection{Lower bounds on $\|K\|_2$} \label{ssec:lower}

Let us consider the problem of bounding $\|K\|_2$ from below, which, by Theorem~\ref{thm:l2-2}, would quantify how slowly the chain converges when $K(\cdot,\cdot)$ is reversible.
The problem is not as frequently studied as that of bounding $\|K\|_2$ from above, but it has been examined in the context of some important MCMC algorithms \citep{johndrow2018mcmc,chewi2021optimal,wu2022minimax,andrieu2022explicit}. 
We provide a concise overview of some of the basic techniques employed in these studies.
Although this problem is most meaningful when $K(\cdot,\cdot)$ is reversible, the results we present would not require the assumption of reversibility.

Let $f \in L^2(\varpi)$ be a non-constant function.
Then $f - \varpi f \in L_0^2(\varpi)$, and $f - \varpi f \neq 0$.
By the Cauchy-Schwarz inequality and the definition of $\|K\|_2$,
\begin{equation} \nonumber
\langle f - \varpi f, K(f - \varpi f) \rangle \leq \|f - \varpi f\|_2 \, \|K(f - \varpi f)\|_2 \leq \|K\|_2 \|f - \varpi f\|_2^2.
%, \quad \|K(f - \varpi f)\|_2 \leq \|K\|_2  \|f - \varpi f\|_2 .
\end{equation}

This yields the following bound:
\begin{theorem} \label{thm:lower}
	Let $f \in L^2(\varpi)$ be a non-constant function.
	Then
	\[
	\|K\|_2 \geq \frac{\langle f - \varpi f, K(f - \varpi f) \rangle}{\|f - \varpi f\|_2^2} = 1 - \frac{\int_{\X^2} [f(x) - f(x')]^2 K(x, \df x') \, \varpi(\df x) }{ 2\|f - \varpi f\|_2^2 }.
	\]
\end{theorem}

Theorem~\ref{thm:lower} can be directly applied to scenarios where $\varpi(\cdot)$ has a simple form.
Consider the Gaussian chain in Example \ref{ex:gaussian}.
For $x = (x[1], \dots, x[p]) \in \mathbb{R}^p$, let $f(x) = x[1]$.
Then $\varpi_p f = 0$, $\|f - \varpi_p f \|_2^2 = 1$, and
\[
\langle f - \varpi_p f, K_{p,\alpha}(f - \varpi_p f) \rangle = \langle f, \alpha f \rangle = \alpha.
\]
By Theorem~\ref{thm:lower}, $\|K_{p,\alpha}\|_2 \geq \alpha$.
Recall that $\alpha$ is in fact the true value of $\|K_{p,\alpha}\|_2$.

Letting $f$ be an indicator function in Theorem~\ref{thm:lower} gives the following (after a bit of calculations), which is very similar to parts of Theorem~\ref{thm:cheeger}.
\begin{corollary} \label{cor:lower-A}
	Let $A \in \B$ be such that $\varpi(A) \in (0,1)$.
	Then
	\begin{equation} \label{ine:lower-A}
	\|K\|_2 \geq 1 - \frac{\int_A \varpi(\df x) K(x, A^c) }{\varpi(A) \varpi(A^c)}.
	\end{equation}
\end{corollary}

Finally, from Corollary~\ref{cor:lower-A} we can immediately derive the result below.

\begin{corollary} \label{cor:lower-A-1}
	Suppose that there exist $A \in \B$ and $\delta \leq 1$ such that $K(x,\{x\}^c) \leq \delta$ for $x \in A$.
	Then, if $\varpi(A) \in (0,1)$,
	\[
	\|K\|_2 \geq 1 - \frac{\delta}{\varpi(A^c)}.
	\]
\end{corollary}

By Corollary \ref{cor:lower-A-1}, to obtain a large lower bound on $\|K\|_2$, we only need to find a set $A$ such that $\varpi(A)$ is small, and that $K(x,\{x\})$ is large when $x \in A$.
Oftentimes, we can take $A$ to be an arbitrarily small neighborhood around a point $x_0$ where $K(x_0, \{x_0\})$ is large.

We my apply Corollary \ref{cor:lower-A-1} to the independent Metropolis Hastings chain from Example \ref{ex:IMH}.
Recall that $s: [0,1] \to (0,\infty)$ is continuous, and $M_s = \sup_{x \in [0,1]} s(x) < \infty$.
Then, for $\varepsilon > 0$, one can find an interval $A_{\varepsilon} = (a_{\varepsilon}, b_{\varepsilon}) \subset [0,1]$ such that $0 < b_{\varepsilon} - a_{\varepsilon} < \varepsilon$, and that $s(x) > M_s - \varepsilon$ for $x \in  (a_{\varepsilon}, b_{\varepsilon})$.
Then, whenever $\varepsilon \in (0, M_s)$, it holds that $\pi_s(A_{\varepsilon}) \leq \varepsilon M_s$, and, for $x \in A_{\varepsilon}$,
\[
K_s(x, \{x\}^c) = \int_0^1 a_s(x,x') \, \df x' = \int_0^1 \min\left\{ 1, \frac{s(x')}{s(x)} \right\} \, \df x' \leq \int_0^1 \frac{s(x')}{M_s - \varepsilon} \, \df x' = \frac{1}{M_s - \varepsilon}.
\]
Hence, by Corollary \ref{cor:lower-A-1}, if $\varepsilon < M_s$ and $\varepsilon < 1/M_s$,
\[
\|K_s\|_2 \geq 1 - \frac{1}{(M_s - \varepsilon)(1 - \varepsilon M_s)}.
\]
Since $\varepsilon$ can be arbitrarily small, we have the bound $\|K_s\|_2 \geq 1-1/M_s$.
In Section \ref{ssec:l2basic}, it was shown that $\|K_s\| \leq 1 - 1/M_s$.
Thus, we may conclude that $\|K_s\|_2 = 1 - 1/M_s$.

Corollary~\ref{cor:lower-A-1} can be used to analyze Markov chains associated with general Metropolis Hastings algorithms.
%In these algorithms, there is a positive probability that the state of the Markov chain remains unchanged in one iteration.
See \cite{brown2022lower} for a detailed discussion on this topic.

There are also results showing the slowness of Markov chains outside the $L^2$ framework.
See, e.g., \cite{roberts2011quantitative}, \cite{wang2022exact}, \cite{brown2022lower}.

\subsubsection{Example: a random walk Metropolis Hastings algorithm}

We now illustrate Theorem~\ref{thm:lower} and Corollary~\ref{cor:lower-A-1} through a semi-toy example, which is a simplified version of a study from \cite{andrieu2022explicit}.

Consider a random walk Metropolis Hastings (RWMH) algorithm on $\mathbb{R}^p$ targeting the $p$-dimensional standard normal distribution, which we denote by $\varpi_p(\cdot)$.
Let $\sigma$ be a positive constant.
Given the current state $x \in \mathbb{R}^p$, the RWMH algorithm proceeds as follows.
Draw $X'$ from the $\mbox{N}_p(x, \sigma^2 I_p )$ distribution, and call its realization $x'$.
Let 
\[
a(x,x') = \min\left\{ 1, \frac{\exp( -\|x'\|^2/2) }{ \exp( -\|x\|^2/2 )} \right\}.
\]
With probability $a(x,x')$, set the next state to $x'$, and with probability $1 - a(x,x')$, set the next state to~$x$.
The underlying Markov chain is reversible with respect to $\varpi_p(\cdot)$.
Denote the Mtk of this algorithm by $T_{p,\sigma}$.
Then, for $x \in \mathbb{R}^p$ and $A \in \B$,
\[
\begin{aligned}
	T_{p,\sigma}(x,A) =& \int_A \frac{1}{(2\pi \sigma^2)^{p/2}} \exp \left( - \frac{\|x'-x\|^2}{2 \sigma^2} \right) a(x,x') \, \df x' +  \\
	& \int_{\mathbb{R}^p} \frac{1}{(2\pi \sigma^2)^{p/2}} \exp \left( - \frac{\|x'-x\|^2}{2 \sigma^2} \right) [1 - a(x,x')] \, \df x' \, \ind_{x \in A}.
\end{aligned}
\]

We now apply Theorem~\ref{thm:lower} to bound $\|T_{p,\sigma}\|_2$ from below.
For $x = (x[1], \dots, x[p]) \in \mathbb{R}^p$, let $f(x) = x[1]$.
Then $\|f - \varpi_p f\|_2^2 = 1$.
Moreover, for $x \in \mathbb{R}^p$,
\[
\begin{aligned}
	\int_{\mathbb{R}^p} [f(x) - f(x')]^2 \, T_{p,\sigma}(x, \df x') &= \int_{\mathbb{R}^p}  \frac{[f(x) - f(x')]^2}{(2\pi \sigma^2)^{p/2}} \exp \left( - \frac{\|x'-x\|^2}{2 \sigma^2} \right) a(x,x') \, \df x' \\
	&\leq \int_{\mathbb{R}^p}  \frac{[f(x) - f(x')]^2}{(2\pi \sigma^2)^{p/2}} \exp \left( - \frac{\|x'-x\|^2}{2 \sigma^2} \right) \df x' = \sigma^2.
\end{aligned}
\]
By Theorem~\ref{thm:lower}, $\|T_{p,\sigma}\|_2 \geq 1 - \sigma^2/2$.

Another lower bound on $\|T_{p,\sigma}\|_2$ can be obtained through Corollary~\ref{cor:lower-A-1}.
Elementary calculations show that
\[
\begin{aligned}
	T_{p,\sigma}(0, \{0\}^c) &= \int_{\mathbb{R}^p} \frac{1}{(2\pi \sigma^2)^{p/2}} \exp \left( - \frac{\|x'-0\|^2}{2 \sigma^2} \right) a(0,x') \, \df x' \\
%	&= \int_{\mathbb{R}^p} \frac{1}{(2\pi \sigma^2)^{p/2}} \exp \left( - \frac{\|x'\|^2}{2 \sigma^2} - \frac{\|x'\|^2}{2} \right) \df x' \\
	&= \frac{1}{(\sigma^2 + 1)^{p/2}}.
\end{aligned}
\]
Moreover, one can verify that $x \mapsto T_{p,\sigma}(x, \{x\}^c)$ is a continuous function.
Hence, there exists a sequence of open neighborhoods of 0, say, $A_1, A_2, \dots$, such that $\lim_{n \to \infty} \varpi_p(A_n) = 0$, and that 
\[
\lim_{n \to \infty} \sup_{x \in A_n} T_{p,\sigma}(x,\{x\}^c) = \frac{1}{(\sigma^2 + 1)^{p/2}}.
\]
By Corollary~\ref{cor:lower-A-1},
\[
\|T_{p,\sigma}\|_2 \geq 1 - \lim_{n \to \infty} \frac{\sup_{x \in A_n} T(x,\{x\}^c)}{1 - \varpi_p(A_n)} = 1 - \frac{1}{(\sigma^2+1)^{p/2}}.
\]

Combining the two bounds, we see that
\begin{equation} \label{ine:metro-lower}
\|T_{p,\sigma}\|_2 \geq 1 - \min \left\{ \frac{\sigma^2}{2}, \frac{1}{(\sigma^2+1)^{p/2}} \right\}.
\end{equation}

In practice, we may tune $\sigma$ in the hopes of making $\|T_{p,\sigma}\|_2$ small.
But how small can $\|T_{p,\sigma}\|_2$ go?
We can partially answer this by minimizing the lower bound in \eqref{ine:metro-lower}.
The bound is minimized when $\sigma = \sigma_p$, where $\sigma_p$ satisfies
\[
\frac{\sigma_p^2}{2} = \frac{1}{(\sigma_p^2+1)^{p/2}}.
\]
It can be shown that, when $p$ is sufficiently large, $1/p \leq \sigma_p^2 \leq 2 (\log p) / p$, and
\[
1 - \min \left\{ \frac{\sigma_p^2}{2}, \frac{1}{(\sigma_p^2+1)^{p/2}} \right\} \geq 1 - \frac{\log p}{p}.
\]
This means that, regardless of how $\sigma$ is chosen, $\|T_{p,\sigma}\|_2$ is always lower bounded by $1 - (\log p)/p$ for large values of~$p$.

When $\sigma^2 = 1/p$, \eqref{ine:metro-lower} implies that $1 - \|T_{p,\sigma}\|_2 \leq  1/(2p)$.
In this case, the bound gives the correct order when $p \to \infty$.
Indeed, using isoperimetric inequalities, it is possible to establish a lower bound on $1 - \|T_{p,\sigma}\|_2$ that is also of the order $1/p$ when $\sigma^2 = 1/p$.
See \cite{andrieu2022explicit}, who studied RWMH algorithms targeting general distributions with strongly log-concave densities.

\section{Other methods} \label{sec:others}

We end this chapter by listing some other important methods for constructing convergence bounds.

The canonical path technique is a powerful tool for analyzing Markov chains taking values in a discrete state space \citep{diaconis1991geometric,sinclair1992improved,yang2016computational}.

The transition law of a Markov chain can be written into a random function, and convergence bounds may be formed by studying the local contractive behavior of this function \citep{steinsaltz1999locally,jarner2001locally,qin2020wasserstein,qu2023computable}.

The convergence properties of a Markov chain with a complicated transition law can be studied by comparing it to a simpler Markov chain or process \citep{jones2014convergence,latuszynski2014convergence,pillai2014ergodicity,dalalyan2017theoretical,andrieu2018uniform,rudolf2018perturbation,ascolani2023dimensionfree}.
In particular, the optimal scaling framework provides a unique perspective for studying the properties of a high-dimensional Metropolis-Hastings algorithm by relating it to a certain diffusion process \citep{gelman1997weak,atchade2011towards,pillai2012optimal,yang2020optimal}.

One can also decompose an intricate transition law into simpler components \citep{madras2002markov,jerrum2004elementary,guan2007small,woodard2009conditions,ge2018simulated,qin2023spectral}.
Related to this approach, techniques based on spectral independence and stochastic localization have recently received an increasing amount of attention \citep{anari2021spectral,chen2021optimal,chen2021rapid,chen2022localization,feng2022rapid,qin2022spectral}.

Finally, some MCMC algorithms can be conceptualized as certain deterministic optimization algorithms over a space of distributions.
These algorithms can be analyzed using the theory of gradient flows.
See \cite{cheng2018convergence}, \cite{durmus2019analysis}, and references therein.

%\subsection{References}
%\label{ss:sub2}
%
%This file uses BibTex to generate references.  See the source code
%for how to cite references in the text if you are unfamiliar with
%BibTeX.   For example, you might want to do one of the following
%things 
%\begin{enumerate}
%\item[] \citet{roberts-rosenthal-langevin}
%\item[] \cite{roberts-rosenthal-langevin}
%\item[] \citep{roberts-rosenthal-langevin}
%\item[] \citet{roberts-rosenthal-langevin,geyer-thompson-tempering}
%\item[] \citep{roberts-rosenthal-langevin,geyer-thompson-tempering}
%\item[] \pcite{roberts-rosenthal-langevin}
%
%\end{enumerate}
%
%\subsection{Theorems}
%\label{ss:sub3}
%
%% Theorems and proof environments are loaded from AMS libraries. 
%\begin{theorem} \label{thm:trig}
%$\sin^{2} \theta + \cos^{2} \theta = 1$
%\end{theorem}
%
%\begin{proof}
%Here is how it goes\ldots
%\end{proof}

%  If you want to add acknowledgments, you can do it
%  in a section like this.

\section*{Acknowledgments}
This work was partially supported by the NSF.
The author thanks Austin Brown, Galin L. Jones, and Youngwoo Kwon for their helpful comments.

\bibliographystyle{apalike} 
\bibliography{handbook-QQ}

\end{document}